\documentclass[sn-mathphys]{sn-jnl}% Math and Physical Sciences Reference Style
\jyear{2022}%

\usepackage{color}
\usepackage{multicol}
\usepackage{amsthm}
\usepackage{amsmath}
\usepackage{subcaption}
\usepackage{graphics}
\usepackage{pdflscape}
\theoremstyle{thmstyleone}%
\newtheorem{thm}{Theorem}%  meant for continuous numbers

\newtheorem{prop}{Proposition}

\theoremstyle{thmstyletwo}%
\newtheorem*{Exp}{Example:}%
\newtheorem*{rmk}{Remark:}%

\theoremstyle{thmstylethree}%
\newtheorem{definition}{Definition}%
\newtheorem{assmp}{Assumption}

\begin{document}

\title[Polytopic Superset Algorithm for Nonlinear Robust Optimization]{Polytopic Superset Algorithm for Nonlinear Robust Optimization}

%%=============================================================%%
%% Prefix	-> \pfx{Dr}
%% GivenName	-> \fnm{Joergen W.}
%% Particle	-> \spfx{van der} -> surname prefix
%% FamilyName	-> \sur{Ploeg}
%% Suffix	-> \sfx{IV}
%% NatureName	-> \tanm{Poet Laureate} -> Title after name
%% Degrees	-> \dgr{MSc, PhD}
%% \author*[1,2]{\pfx{Dr} \fnm{Joergen W.} \spfx{van der} \sur{Ploeg} \sfx{IV} \tanm{Poet Laureate} 
%%                 \dgr{MSc, PhD}}\email{iauthor@gmail.com}
%%=============================================================%%

\author*[1]{\fnm{Bowen} \sur{Li}}\email{bowen.li@anl.gov}

\author[1]{\fnm{Kibaek} \sur{Kim}}\email{kimk@anl.gov}
% \equalcont{These authors contributed equally to this work.}

\author*[1]{\fnm{Sven} \sur{Leyffer}}\email{leyffer@anl.gov}
% \equalcont{These authors contributed equally to this work.}

\affil[1]{\orgdiv{Mathematics and Computer Science Division}, \orgname{Argonne National Laboratory}, \orgaddress{ \city{Lemont}, \state{IL}, \country{USA}}}

% \affil[2]{\orgdiv{Department}, \orgname{Organization}, \orgaddress{\street{Street}, \city{City}, \postcode{10587}, \state{State}, \country{Country}}}

% \affil[3]{\orgdiv{Department}, \orgname{Organization}, \orgaddress{\street{Street}, \city{City}, \postcode{610101}, \state{State}, \country{Country}}}

%%==================================%%
%% sample for unstructured abstract %%
%%==================================%%

\abstract{Nonlinear robust optimization (NRO) is widely used in different applications, including energy, control, and economics, to make robust decisions under uncertainty. One of the classical solution methods in NRO is an outer approximation method that iteratively solves a sample-based nonlinear problem and updates the sample set by solving an auxiliary problem subject to the uncertainty set. Although it guarantees convergence under certain assumptions, its solution iterates are generally infeasible in the original NRO problem, and it provides only a lower bound on the optimal objective value. We propose a new algorithm for a class of NRO problems that generates feasible solution iterates and provides both lower and upper bounds to the optimal objective value. In each iteration, the algorithm solves the reformulation of an NRO subproblem with respect to the polytopic supersets of the original uncertainty set and uses a cutting plane method to improve the supersets over iteration. If the NRO subproblem is infeasible, we provide a feasibility restoration step to detect whether the original NRO problem is infeasible or construct a new superset to restore the feasibility of the NRO subproblem. Further, we prove that our superset algorithm converges to the optimal solution of the original NRO problem. In numerical studies, we use application instances from portfolio optimization and production cost minimization and compare the performance between the superset algorithm and an outer approximation method called Polak's algorithm. Our result shows that the superset algorithm is more advantageous than Polak's algorithm when the number of robust constraints is large.}


%%\pacs[JEL Classification]{D8, H51}

%%\pacs[MSC Classification]{35A01, 65L10, 65L12, 65L20, 65L70}

\maketitle

\section{Introduction}\label{sec:intro}

We consider the following structured nonlinear robust optimization (NRO) problem, 
\begin{align}\label{eq:original_multi}
        \min_{x}\ f(x)\quad\mbox{s.t. } u_i^\top h_i(x)\leq b_i(x),\ \forall u_i\in\mathcal{U}_i,\ i=1,2,...,I,
\end{align}
where $x\in\mathbb{R}^n$ are the decision variables, $f(x):\mathbb{R}^n\rightarrow\mathbb{R}$ is the objective function, and $I$ denotes the number of robust constraints. For each constraint $i$, $u_i\in\mathbb{R}^{p_i}$ are the uncertain parameters of dimension $p_i$; $\mathcal{U}_i$ are the uncertainty sets; and $h_i(x):\mathbb{R}^n\rightarrow\mathbb{R}^{p_i}$ and $b_i(x):\mathbb{R}^n\rightarrow\mathbb{R}$ are the nonlinear functions on $x$.

For clarity and simplicity, we  consider only the following single-constrained NRO for the remainder of the paper: 
\begin{align}\label{eq:original}
        \min_{x}\ f(x)\quad\mbox{s.t. } u^\top h(x)\leq b(x),\ \forall u\in\mathcal{U},
\end{align}
where $\mathcal{U}:=\{u\in\mathbb{R}^{p}:\ g(u)\leq 0\}$ is the uncertainty set governed by function $g:\mathbb{R}^{p}\rightarrow\mathbb{R}^{m}$, where $m$ represents the number of constraints in uncertainty set $\mathcal{U}$. We denote the components of $g$ by $g_{j}:\mathbb{R}^{p}\rightarrow\mathbb{R}$ with $j\in J=\{1,...,m\}$. We also define the feasible set $\mathcal{F}:=\{x\in\mathbb{R}^n:\ u^\top h(x)\leq b(x),\ \forall u\in\mathcal{U}\}$. It is straightforward to extend the developed results of \eqref{eq:original} to the multiple-constrained NRO \eqref{eq:original_multi} or to include general nonlinear constraints. 

\begin{assmp}\label{asmp:form} 
We make the following assumptions for the NRO problem \eqref{eq:original}:
\begin{enumerate}
       \item For all $j\in J$, $g_{j}(u)$ are convex and continuously differentiable in $u$.
       
       \item The functions $f(x)$, $h(x)$, and $b(x)$ are continuously differentiable in $x$.

    % \item The set $\mathcal{U}$ are compact and satisfy Slater's constraint qualification, namely, $\exists u\in\mathbb{R}_{p}:\ g(u)<0$.
    % %GAIL - earlier you  used $\mathcal{U}$ in  the plural and then in the  singular in the same  sentence,m and here again it  is  plural. Should it be singulr?
    
    \item The uncertainty set $\mathcal{U}$ is compact, nonempty, and the feasible set $\mathcal{F}$ is bounded.

\end{enumerate}
\end{assmp}

\begin{rmk}\label{rmk:assmp}
From Assumption~\ref{asmp:form}, because $g_j(u)$ are convex, we can efficiently check if set $\mathcal{U}$ is empty or not. If $\mathcal{U}=\emptyset$, then we have no robust constraint in \eqref{eq:original}. Otherwise, we conclude that $\mathcal{U}$ is a convex set; the gradient functions $\nabla g_j(u)$ are Lipschitz continuous in $u$ for all $j=1,...,m$. The set $\mathcal{F}$ is closed because it is an intersection of an infinite number of closed sets. From closedness and boundedness of $\mathcal{F}$, we further conclude that $\mathcal{F}$ is a compact set, which implies that \eqref{eq:original} has a solution if $\mathcal{F}$ is nonempty. We will discuss more about the feasibility issue in Section~\ref{rmk:feas}.
\end{rmk}

NRO has many important applications, of which we select two as our motivation for this paper. One instance of NRO is in portfolio optimization \cite{port1,port2}. This problem maximizes the risk-adjusted expected return under the uncertainty of the mean and covariance of the asset returns by finding the optimal asset allocation. Different models of the uncertainty set are discussed in \cite{port2}. The dimension of the uncertain parameter is quadratic in the number of assets. Another example of NRO arises in production cost minimization \cite{prod_cost}. Under the uncertainty of price, this problem minimizes the production cost as well as the cost associated with production ramping while satisfying the daily demand requirements and ramping limitations. The number of robust constraints is proportional to the scheduling horizon. All the conditions in Assumption~\ref{asmp:form} are satisfied by these numerical instances.

Different iterative methods have been developed to solve NRO and can be broadly categorized into discretization methods \cite{dis1,tut4}, exchange methods \cite{outer1, polak1997,polak_boyd,polak3}, and local reduction methods \cite{tut1, tut3, tut5}. In this work, our benchmark is an outer approximation method called Polak's algorithm \cite{polak1997, polak_boyd} in the category of the exchange methods. Polak's algorithm iteratively solves a sample-based subproblem that approximates the uncertainty set $\mathcal{U}$ of the original NRO problem \eqref{eq:original} by a finite sample set. Because the finite sample set is a subset of the uncertainty set $\mathcal{U}$, the sample-based subproblem can be seen as an outer approximation of \eqref{eq:original}. As the iteration progresses, new sample points are obtained by solving a worst-case constraint violation problem and are accumulated to be new constraints of the next-iteration subproblem. Before termination, the resulting solution iterates of Polak's algorithm are not feasible in the original NRO problem \eqref{eq:original} and  provide only a lower bound on the optimal objective value of \eqref{eq:original}. 
% Because of this outer approximation nature, Polak's algorithm only detects the infeasibility of \eqref{eq:original} when the worst-case sample points are obtained. 

The major benefit of the affine relationship between the constraint $h(x)$ and uncertainty parameter $u$ in \eqref{eq:original} is the existence of a reformulation approach.  If we approximate the uncertainty set $\mathcal{U}$ with a polytopic superset and apply linear programming duality to the robust constraints, then the resulting NRO subproblem can be exactly reformulated into a finite-dimensional nonlinear problem similar to \cite{nonlin1, rob_book}, which can be solved by off-the-shelf nonlinear optimization solvers. Hence, we propose a superset algorithm that iteratively solves the reformulation of the NRO subproblem with polytopic supersets and aims to improve the solution quality by gradually constructing better supersets. We show that our proposed superset algorithm follows an  iterative structure similar to Polak's algorithm \cite{polak_boyd} but generates feasible iterates.

% , generate solution iterates feasible in the original NRO, detects infeasibility and provides both lower and upper bounds to the optimal objective value.  

Our contribution is the development of a superset algorithm that iteratively solves the reformulation of an NRO subproblem with polytopic supersets of the original uncertainty set. We propose different cutting plane methods to improve the supersets iteratively. Compared with Polak's algorithm \cite{polak1997,polak_boyd}, the solution iterates of our algorithm are feasible in the original NRO problem and provide both upper and lower bounds to the optimal objective value. We show that with the proposed cutting plane methods, the solution of the NRO subproblem with polytopic supersets converges to the optimal solution of the original NRO problem. In addition, we provide a feasibility restoration algorithm to restore the feasibility if the initial supersets are overly conservative or detect whether the original NRO problem is infeasible. To demonstrate the computation performance, we compare the superset algorithm with Polak's algorithm with test instances from portfolio optimization and production cost minimization. We show that the superset algorithm outperforms Polak's algorithm in most of the test cases, especially when the number of robust constraints is large. 

The outline of this paper is as follows. Section~\ref{sec:background} discusses  theoretical background about general NRO problems. Section~\ref{sec:new} discusses the benefits of using polytopic supersets and gives a detailed description and theoretical analysis of our proposed superset algorithm. Section~\ref{sec:conv} gives the convergence analysis, and Section~\ref{sec:numerical} shows the numerical results of comparing the superset algorithm with Polak's algorithm in portfolio optimization and production cost minimization. Section~\ref{sec:conclusion} summarizes our work and briefly discusses future plans. 

\section{Background}\label{sec:background}
In this section we give some theoretical results from the literature as background. For simplicity, we define the following constants and notations for the rest of the paper. First, for the convergence analysis in Section~\ref{sec:conv}, we define a unified Lipschitz constant $L$ for all the gradient functions $\nabla g_j(u)$ and radius $R$ for the bounded set $\mathcal{U}$. Second, we define the following notation rules: 
For a finite set $U$, $\lvert U \rvert$ denotes its number of elements; and for a closed set $\mathcal{U}$, $\partial\mathcal{U}$ denotes its boundary. We define $[x]^+=\max\{x,0\}$. For a vector $x$, we define $x\nleq 0$ if there exists a component $x_i$ such that $x_i>0$. 

\subsection{General Properties of Robust Optimization}\label{sec:cq}
Here we recall some of the main theoretical results regarding the optimal solution of \eqref{eq:original}. Given $x\in\mathcal{F}$, we define the active set $\mathcal{A}(x):=\{u\in\mathcal{U}:\ u^\top h(x)=b(x)\}$ and Jacobian of $h(x)$ as $\nabla h(x)=[\nabla h_1(x),...,\nabla h_p(x)]$. Next, we give the definition of the Mangasarian--Fromovitz constraint qualification (MFCQ) of \eqref{eq:original}:
\begin{definition}[Mangasarian--Fromovitz Constraint Qualification \cite{emfcq1,emfcq2, emfcq3}]\label{def:emfcq}
We say that $x^*$ satisfies the MFCQ if there exists $s\in\mathbb{R}^n$ such that  
\begin{align*}
        u^\top h(x^*)-b(x^*)+s^\top \left(\nabla h(x^*) u-\nabla b(x^*)\right)<0,\ \forall u\in \mathcal{U}.
\end{align*}    
\end{definition}
Now with MFCQ we give the first-order condition \cite{tut1, tut2, sven_report}  as follows. 
\begin{thm}\label{thm:opt1}
Suppose $x^*\in\mathcal{F}$ satisfies MFCQ. If $x^*$ is a local minimizer of \eqref{eq:original}, then there exist a finite subset $A^*\subseteq \mathcal{A}(x^*)$ and multipliers $\lambda^*_u\geq 0$ for each $u\in A^*$  such that
\begin{align}\label{eq:kkt1}
    \nabla f(x^*)+\sum_{u\in A^*}\lambda^*_u\left(\nabla h(x^*) u-\nabla b(x^*)\right)=0.
\end{align}
\end{thm}

Theorem~\ref{thm:opt1} motivates the error measure $errm(x,\lambda_u,A)$ to evaluate the first-order condition error given solution estimate $x$, a finite set $A$, and the corresponding multipliers $\lambda_u$ for all $u\in A$:

\begin{subequations}\label{eq:robkkt_err}
\begin{align}
    errm(x,\lambda_u,A)&=\left\|\nabla f(x) + \sum_{u\in {A}}\lambda_u\left(\nabla h(x) u -\nabla b(x)\right)\right\|_2 \label{eq:7a}\\
    &+ \left[\max_{u\in\mathcal{U}}u^\top h(x) - b(x)\right]^+  \label{eq:7b}\\
    &+ \sum_{u\in{A}}\ [-\lambda_u]^+ \label{eq:7c}\\
    &+ \sum_{u\in{A}}\ \lvert u^\top h(x) - b(x)\rvert \ + \sum_{u\in {A}}\ \|u-\text{proj}_{\mathcal{U}}(u)\|_2 \label{eq:7d}.
\end{align}
\end{subequations}
The function $\text{proj}_{\mathcal{U}}()$ can be any projection operator to the set $\mathcal{U}$. Without loss of generality, we use the Euclidean projection and define
\begin{align}\label{eq:eud}
    \text{proj}_{\mathcal{U}}(x) = \text{argmin}_{u\in\mathcal{U}}\ \|u-x\|_2^2.
\end{align}
Comparing \eqref{eq:robkkt_err} with Theorem~\ref{thm:opt1}, we observe that \eqref{eq:7a} corresponds to the first-order condition \eqref{eq:kkt1}, \eqref{eq:7b} corresponds to primal feasibility, \eqref{eq:7c} corresponds to the non-negativity of the Lagrangian multipliers, and \eqref{eq:7d} checks the activity condition $A\subseteq\mathcal{A}(x)$. We note that $A\subseteq\mathcal{A}(x)$ if and only if \eqref{eq:7d} equals zero.

\subsection{Polak's Algorithm}
As a benchmark, we also consider an outer approximation method, namely, Polak's algorithm \cite{sven_report,polak1997, polak_boyd}. The algorithm is based on a sample-based outer approximation of \eqref{eq:original}. At iteration $k$, given a finite sample set $U_{k}\subseteq\mathcal{U}$, the algorithm solves the following finite-dimensional nonlinear problem:
\begin{align}\label{eq:Polak_fin}
        \min_{x}\ f(x)\quad \mbox{s.t. } u^\top h(x)\leq b(x),\ \forall u\in U_{k}.
\end{align}
Because we approximate $\mathcal{U}$ with a finite sample set $U_{k}$, \eqref{eq:Polak_fin} is a relaxation of \eqref{eq:original}. The solution of \eqref{eq:Polak_fin} is not guaranteed to be feasible in \eqref{eq:original} and provides only a lower bound on the optimal objective value of \eqref{eq:original}. In each iteration $k$, to improve the current solution $x_k$, we solve the following problem:
\begin{align}\label{eq:feas_err}
        \max_{t, u\in\mathcal{U}}\ t \quad \mbox{s.t. } u^\top h(x_k) - b(x_k)\geq t.
\end{align}
Let $\hat{u}_{k}$ and $t_{k}$ be the solution and optimal objective value of \eqref{eq:feas_err}, respectively. Clearly, $t_{k} \leq 0$ if and only if $x_k\in\mathcal{F}$. When $t_{k}>0$, it corresponds to the worst-case constraint violation at uncertainty realization $\hat{u}_{k}$. In this case, $\hat{u}_{k}$ will be added to the sample set $U_k$ and generates a new constraint in the subproblem of the next iteration. As the sample size increases, the algorithm terminates when $t_{k}\leq \epsilon$ for some tolerance $\epsilon>0$. The algorithmic steps are summarized in Algorithm~\ref{alg:Polak}.

% The full algorithm is given below.

\begin{algorithm}[H]
% \begin{algorithmic}[1]
\begin{algorithmic}
\State {\bf initialization:} Set $k\gets 0$; initialize $U_{0}\gets\emptyset$ and some $\epsilon>0$.
    
\Repeat
\State {\bf sample-based subproblem:} Let $x_{k}$ be the solution of \eqref{eq:Polak_fin}.

\State {\bf update sample set:} Let $\hat{u}_{k}$ and $t_{k}$ be the solution of \eqref{eq:feas_err}.

\State Set $U_{k+1}\gets U_{k}\cup \{\hat{u}_{k}\}$.

\State Set $k\gets k+1$.
   
\Until{$t_{k-1}\leq \epsilon$.}
    
\State {\bf return} $x_{k-1}$.
    
\caption{Polak's Algorithm}\label{alg:Polak}
\end{algorithmic}
\end{algorithm}
    
Note that when $k=0$, \eqref{eq:Polak_fin} can be unbounded. In practical applications, however,  deterministic constraints or the inclusion of a nominal sample can resolve this issue without loss of generality. By applying the proof from \cite{polak_boyd} to \eqref{eq:original}, Algorithm~\ref{alg:Polak} is guaranteed to converge.  

\section{Polytopic Superset Algorithm}\label{sec:new}

In this section we propose a new polytopic superset algorithm and discuss its details. Specifically, in Section~\ref{sec:sketch} we define a polytopic superset NRO; and, based on this, we give an algorithm sketch to the superset algorithm. In Section~\ref{sec:ref} we develop a reformulation approach given the affine relationship between $u$ and $h(x)$ in \eqref{eq:original} to efficiently solve the polytopic superset NRO. In Section~\ref{sec:proj_cuts}, we discuss the cutting plane methods that we use to improve the polytopic superset NRO. In Section~\ref{sec:termination} we extend the error measure \eqref{eq:robkkt_err} to the proposed polytopic superset algorithm and use it as the termination criteria. In Section~\ref{rmk:feas} we develop a feasibility restoration step to either detect whether \eqref{eq:original} is infeasible or construct an appropriate superset that guarantees the feasibility of the polytopic superset NRO. In Section~\ref{sec:full_alg} we give the full algorithmic steps of the polytopic superset algorithm and provide some theoretical analysis.  

\subsection{Sketch of a Polytopic Superset Algorithm}\label{sec:sketch}

In each iteration $k$, we define the following NRO subproblem using a polytopic superset $\mathcal{S}_{k}\supseteq\mathcal{U}$ of the current iteration: 
\begin{align}\label{eq:supsub}
        \min_{x}\ f(x)\quad\mbox{s.t. } u^\top h(x)\leq b(x),\ \forall u\in\mathcal{S}_{k}.
\end{align}

Based on \eqref{eq:supsub}, we propose the following superset algorithm. 

\begin{algorithm}[H]
\begin{algorithmic}
    
\State {\bf initialization:} Set $k\gets0$; initialize a box superset $\hat{\mathcal{S}}\supseteq\mathcal{U}$ and some $\epsilon>0$.

\State {\bf feasibility restoration:} Detect if \eqref{eq:original} is infeasible, or construct a new superset $\mathcal{S}_0$ from $\hat{\mathcal{S}}$ that guarantees the feasibility of \eqref{eq:supsub}.

\Repeat

\State {\bf superset subproblem:} Let $x_k$  be the solution of \eqref{eq:supsub}.
      
\State {\bf update superset:} Update $\mathcal{S}_{k}$ by adding cutting planes.

\State Set $k\gets k+1$.

\Until $x_{k-1}$ is optimal to \eqref{eq:original} for some tolerance $\epsilon$.       
     
\State {\bf return} $x_{k-1}$.
\end{algorithmic}
\caption{Sketch of the Superset Algorithm}\label{alg:sw_sketch}
\end{algorithm}
We initialize the algorithm using a box superset denoted  $\hat{\mathcal{S}}$ that contains the compact uncertainty set $\mathcal{U}$. By construction, the set $\hat{\mathcal{S}}$ is compact. As the first step, we use feasibility restoration to either provide a certificate of infeasibility to \eqref{eq:original} or construct a superset $\mathcal{S}_0$ based on $\hat{\mathcal{S}}$ such that \eqref{eq:supsub} is guaranteed to be feasible when $k=0$ (see Section~\ref{rmk:feas}). Over the iteration, we keep solving the NRO subproblem \eqref{eq:supsub} and update the superset with cutting planes generated using $x_k$ (see Section~\ref{sec:proj_cuts}). These cutting planes help reduce the gap between the superset $\mathcal{S}_k$ and the uncertainty set $\mathcal{U}$. The algorithm terminates when $x_k$ is optimal to \eqref{eq:original} with some given tolerance $\epsilon$.

Because $\mathcal{S}_{k}\supseteq \mathcal{U}$ for all $k$, all iterates of \eqref{eq:supsub} are also feasible in \eqref{eq:original} and provide an upper bound on the optimal objective value of \eqref{eq:original}. In other words, the NRO subproblem \eqref{eq:supsub} can be seen as an inner approximation of \eqref{eq:original}, which is distinct from the classical outer approximation methods \cite{polak1997, polak_boyd, polak3}. In Section~\ref{sec:full_alg} we will discuss how the superset algorithm also provides a lower bound on the optimal objective value of \eqref{eq:original} through the Euclidean projection used in the termination criteria.

\subsection{Reformulation of the Polytopic Superset NRO}\label{sec:ref}
In this section we develop a reformulation approach to efficiently solve \eqref{eq:supsub}. First, we model the polytopic superset for each iteration $k$ using linear inequalities. We define $\mathcal{S}_{k}:=\{u:\ B_{k} u\leq d_{k}\}$ with linear coefficients $B_{k}\in\mathbb{R}^{b_{k} \times p}$ and $d_{k}\in \mathbb{R}^{b_{k}}$, where $b_{k}$ denotes the number of linear inequality constraints in $\mathcal{S}_{k}$. In our algorithm, $B_k u\leq d_k$ corresponds to a set of supporting hyperplanes for the convex uncertainty set $\mathcal{U}$, and we show in Section~\ref{sec:proj_cuts} how to update $\mathcal{S}_k$. Because of the affine relationship between $h(x)$ and $u$ in \eqref{eq:supsub}, we have the following proposition by applying the linear programming duality to the robust constraint.

\begin{prop}\label{prop:robref}
  Given polytopic supersets $\mathcal{S}_{k}:=\{u\in\mathbb{R}^{p}:\ B_{k}u\leq d_{k}\}$, problem \eqref{eq:supsub}
can be reformulated into the following finite-dimensional nonlinear optimization problem:
\begin{subequations}\label{eq:nonlin}
\begin{align}
        \min_{x,\gamma_{k}}\ &f(x)\\
        \mbox{s.t. }&\gamma_{k}^\top d_{k}\leq b(x)\label{eq:non1},\\
        &(B_{k})^\top\gamma_{k}=h(x)\label{eq:non2},\\
        &\gamma_{k}\geq 0.\label{eq:non3}
\end{align}
\end{subequations}
\end{prop}
\begin{proof}
    The proof can be obtained by generalizing the reformulation approach used for a robust linear program with a polytopic uncertainty set \cite{nonlin1, rob_book} to our structured nonlinear problem \eqref{eq:supsub}. First, because $\mathcal{S}_k$ is compact, we have
    \begin{align}\label{eq:dual1_1}
    \max_{u\in \mathcal{S}_k} u^\top h(x_k)\leq b(x_k)\Leftrightarrow\ u^\top h(x_k)\leq b(x_k),\ \forall u\in\mathcal{S}_k.
\end{align}

Next, by linear programming duality, we have 
\begin{align}\label{eq:dual_1}
    \min_{\gamma\geq0, (B_{k})^\top\gamma=h(x_k)}\gamma^\top d_{k} = \max_{u\in \mathcal{S}_k}u^\top h(x_k).
\end{align}

Substituting the robust constraint in \eqref{eq:supsub} in the left-hand side of \eqref{eq:dual_1}, we obtain reformulation \eqref{eq:nonlin}. 
\end{proof}
Our superset algorithm iteratively solves  reformulation \eqref{eq:nonlin} of the NRO subproblem \eqref{eq:supsub} with polytopic supersets. The reformulation is a finite-dimensional nonconvex problem that can be solved by off-the-shelf nonlinear optimization solvers. 
From Proposition~\ref{prop:robref}, \eqref{eq:supsub} and \eqref{eq:nonlin} are equivalent in the sense that they have the same global optimal solution. Unfortunately, even if \eqref{eq:supsub} is convex,  $h(x)$ is convex and $b(x)$ is concave, the reformulation \eqref{eq:supsub} of \eqref{eq:nonlin} is nonconvex. 
However, we show next, that we can at least recover the equivalence of stationary points in this reformulation.
In the next theorem we show that the stationary points of \eqref{eq:nonlin} satisfy the first-order condition \eqref{eq:kkt1} of the infinite-dimensional NRO subproblem \eqref{eq:supsub}. 

% In the next theorem, we show the connection between the KKT conditions of \eqref{eq:nonlin} to the first-order condition \eqref{eq:kkt1} of the NRO subproblem \eqref{eq:supsub}.

\begin{thm}\label{thm:main}
If $(x_k,\gamma_{k})$ satisfies the KKT conditions of \eqref{eq:nonlin}, then $x_k$ also satisfies the first-order condition \eqref{eq:kkt1} (applied to \eqref{eq:supsub}).
\end{thm}
 % and is feasible in \eqref{eq:supsub}

\begin{proof}
% First, we show that $x_k$ is feasible in \eqref{eq:supsub}. Because $(x_k,\gamma_{k})$ is feasible in \eqref{eq:nonlin}, we have 
% \begin{align}\label{eq:ineq1}
%     \gamma_{k}^\top d_{k} \leq b(x_k).
% \end{align} 
% Next, from the linear programming duality, we have the following equality 
% \begin{align}\label{eq:dual}
%     \min_{\gamma\geq0, (B_{k})^\top\gamma=h(x_k)}\gamma^\top d_{k} = \max_{u\in \{B_{k} u\leq d_{k}\}}u^\top h(x_k).
% \end{align}
% Further from \eqref{eq:ineq1} and \eqref{eq:dual}, we conclude that
% \begin{align}\label{eq:dual1}
%     \max_{u\in \{B_{k} u\leq d_{k}\}}u^\top h(x_k)\leq b(x_k)\Leftrightarrow\ u^\top h(x_k)\leq b(x_k),\ \forall u\in\{B_{k} u\leq d_{k}\},
% \end{align}
% which shows that $x_k$ is feasible in \eqref{eq:supsub}.
First, we analyze the KKT conditions for \eqref{eq:nonlin} where $\mu_{k}$, $v_{k}$, and $\nu_{k}$ are the corresponding multipliers for \eqref{eq:non1}, \eqref{eq:non2}, and \eqref{eq:non3}, respectively:
\begin{subequations}
\begin{align}
    &\nabla f(x_k)+\nabla h(x_k) v_{k} -\mu_{k} \nabla b(x_k)=0,\label{stt1}\\
    &-B_{k} v_{k}+\mu_{k} d_{k}-\nu_{k} =0,\label{stt2}\\
    &\mu_{k}\geq0,\ \nu_{k}\geq 0,\ \mu_{k}\left(\gamma_{k}^\top d_{k}-b(x_k)\right)=0,\ \nu_{k}^\top\gamma_{k}=0,\label{stt3}\\
    &\gamma_{k}^\top d_{k}\leq b(x_k),\ \gamma_{k}\geq0,\ B_{k}^\top\gamma_{k}=h(x_k).\label{stt4}
\end{align}
\end{subequations}
When $\mu_{k}>0$, we have that 
\begin{subequations}\label{eq:cons1}
\begin{align}
    \eqref{stt3}&\Rightarrow\ \gamma_{k}^\top d_{k}=b(x_k),\\
    \eqref{stt2}\mbox{ and }\eqref{stt4}&\Rightarrow\ B_{k}(\frac{v_{k}}{\mu_{k}})=d_{k}-\frac{\nu_{k}}{\mu_{k}}\Rightarrow B_{k}(\frac{v_{k}}{\mu_{k}})\leq d_{k}\Leftrightarrow \frac{v_{k}}{\mu_{k}}\in\mathcal{S}_{k}\label{eq:13b}.
\end{align}
\end{subequations}
Further, we have
\begin{align} \label{eq:cons2}
    \eqref{eq:13b}\Rightarrow \gamma_{k}^\top B_{k}(\frac{v_{k}}{\mu_{k}})=\gamma_{k}^\top d_{k}-\gamma_{k}^\top(\frac{\nu_{k}}{\mu_{k}})\Rightarrow (\frac{\nu_{k}}{\mu_{k}})^\top h(x_k)=b(x_k),
\end{align}
where the second implication comes from  \eqref{stt3} and \eqref{stt4}. Further, from \eqref{eq:dual1_1}, \eqref{eq:cons1}, and \eqref{eq:cons2}, we conclude that $u^\top h(x_k)=b(x_k)$ is a supporting hyperplane to $\mathcal{S}_{k}$ and $\frac{v_{k}}{\mu_{k}}\in\partial \mathcal{S}_{k}$.

Next, we show that $\mu_{k}=0$ implies $v_{k}=0$ by contradiction. Suppose $v_{k}\neq 0$. Because $\mathcal{S}_{k}$ is compact by construction, we have that $B_{k}$ has full-column rank.  From \eqref{stt2} and \eqref{stt3}, we then have $B_{k}v_{k}=-\nu_{k}\leq 0$ (since $\mu_{k}=0$). Because $v_{k}\neq 0$ and $B_{k}$ is full column rank, we conclude that $\nu_{k}\neq 0$. This implies that for any $\hat{u}\in\mathcal{S}_{k}$, we can pick $\alpha>0$ arbitrarily large such that $\hat{u}+\alpha v_{k}\in\mathcal{S}_{k}$ because $B_{k}(\hat{u}+\alpha v_{k})=B_{k}\hat{u} -\alpha\nu_{k}\leq d_{k}$, which contradicts the compactness of $\mathcal{S}_{k}$. Thus, we have proved that $\mu_{k}=0$ implies $\nu_{k}=0$.

Now, we can rewrite \eqref{stt1} as follows:
\begin{subequations}\label{eq:cons3}
\begin{align}
    &\nabla f(x_k)+ \mu_{k}\left(\nabla h(x_k) \left(\frac{v_{k}}{\mu_{k}}\right) -\nabla b(x_k)\right)=0,\quad\mbox{if $\mu_k>0$},\\
    & \nabla f(x_k) = 0,\quad\mbox{if $\mu_k = 0$}.
\end{align}
\end{subequations}
From \eqref{eq:cons3} we conclude that $x_k$ satisfies the first-order condition of \eqref{eq:supsub}.
% regardless of the strict positivity of $\mu_k$. 
\end{proof}

We observe that by substituting the optimal multipliers of \eqref{eq:nonlin} into  \eqref{eq:cons3}, we get the first-order condition of \eqref{eq:supsub}. Depending on the strict positivity of $\mu_k$, however, we denote the solution pair of \eqref{eq:nonlin} as $(x_k, \lambda_k, u_k)$, where 
\begin{align}\label{eq:pair}
   \lambda_k= 
\begin{cases}
    \mu_{k}>0,& \text{if } \mu_k>0\\
    0,              & \text{otherwise}
\end{cases},\quad  u_k= 
\begin{cases}
    \frac{v_{k}}{\mu_{k}},& \text{if } \mu_k>0\\
    \text{any } u\in\mathcal{U},              & \text{otherwise.}
\end{cases}    
\end{align}

% instead of using the optimal multipliers of \eqref{eq:nonlin},

\subsection{Projection and Cuts}\label{sec:proj_cuts}
To improve the solution of \eqref{eq:nonlin} over iteration, we use cutting planes to reduce the size of $\mathcal{S}_k$ and remove the violated uncertainty realization in the current solution pair. At iteration $k$, we have the solution pair $(x_{k},\lambda_{k},u_{k})$ from the reformulation \eqref{eq:nonlin}. If $u_{k}\notin\mathcal{U}$,  we add cutting planes that are valid for the uncertainty set $\mathcal{U}$, to remove the violated point $u_{k}$ from $\mathcal{S}_{k}$ and reduce the gap between the superset $\mathcal{S}_k$ and the uncertainty set $\mathcal{U}$. Next, we will provide different ways to generate cutting planes that exclude such points $u_{k}$. We define the Jacobian of $g(u)$ as $\nabla g(u) = [\nabla g_1(u),...,\nabla g_m(u)]$.

\subsubsection{Kelley's Cutting Plane}
Given a point $u_{k}$ with $g(u_{k})\nleq 0$, we first define  Kelley's cutting plane on $u$ as from \cite{Kelley}
\begin{align}\label{eq:kelleycut}
    \nabla g(u_{k})^\top u \leq \nabla g(u_{k})^\top u_{k}-g(u_{k}).
\end{align}

Because $g$ is convex, $\nabla g(u_{k})^\top (u-u_{k}) + g(u_{k}) \leq g(u)\leq 0$ for all $u\in\mathcal{U}$. We note that $u_{k}$ violates the cut because $g(u_{k})\nleq 0$. 

We also note that the cutting plane \eqref{eq:kelleycut} may not support the uncertainty set $\mathcal{U}$. For example, with uncertainty set $\mathcal{U}:=\{(u_1,u_2)\ \lvert\ u_1^2 + u_2^2\leq 1\}$ with violated $u=(2,0)\notin\mathcal{U}$, the resulting cutting plane is $u_1\leq 1.25$, which does not support $\mathcal{U}$. Next, we provide two alternatives to address this issue.

\subsubsection{Euclidean Projection Cut}
We can generate an alternative cutting plane at Euclidean projection $z_{k}= \text{proj}_{\mathcal{U}}(u_{k})$ that removes $u_{k}$ and supports $\mathcal{U}$. The cutting plane is given by 
\begin{align}\label{eq:eudcut}
    \nabla g(z_{k})^\top u \leq \nabla g(z_{k})^\top z_{k}-g(z_{k}).
\end{align}

Cutting plane \eqref{eq:eudcut} is also valid for $\mathcal{U}$ and $z_{k}\in\partial\mathcal{U}$ because of the Euclidean projection. Hence, \eqref{eq:eudcut} supports $\mathcal{U}$. The exclusion of $u_{k}$ when $u_k\notin\mathcal{U}$ was proved in  Theorem 1 of \cite{eudcut}. 

% To the best of our knowledge, this is the first exclusion proof of such cutting plane \eqref{eq:eudcut}.

% \begin{prop}\label{prop:active cut}
%  There exists $j\in J$ such that $\nabla g_{j}(z_{k})(u_{k}-z_{k})+g_{j}(z_{k})> 0$.
% \end{prop}
% \begin{proof}
%     Define the active set of problem \eqref{eq:eud} as $\hat{A}=\{j\in J\ :\ g_{j}(z_{k})=0\}$. Because $\mathcal{U}$ satisfies Slater's condition, from the KKT condition of \eqref{eq:eud}, we have
%     \begin{align}\label{eq:optcut}
%         2(z_{k}-u_{k})+\sum_{j\in\hat{A}}\lambda_j\nabla g_{j}(z_{k})=0,
%     \end{align}
%     where $\lambda_{j}\geq 0$ for all $j\in\hat{A}$ are the corresponding multipliers. Further because $u_{k}\neq z_{k}$, at least one $\lambda_j>0$. 
    
%     With the KKT condition of \eqref{eq:eud}, we then have that
%     \begin{align*}
%         \sum_{j\in\hat{A}}\lambda_j\nabla g_{j}(z_{k})^\top(u_{k}-z_{k})=2\|u_{k}-z_{k}\|_2^2>0.
%     \end{align*}
%     Hence, we conclude that there exists $j^*\in\hat{A}$ with $\lambda_{j^*}>0$ such that 
%     \begin{align*}
%         \nabla g_{j^*}(z_{k})^\top(u_{k}-z_{k}) = \nabla g_{j^*}(z_{k})^\top(u_{k}-z_{k}) + g_{j^*}(z_{k})>0,
%     \end{align*}
%     Hence, we conclude that cutting plane collection $\nabla g(z_{k})(u-z_{k})+g(z_{k})\leq 0$ excludes $u_{k}$ (i.e., $\nabla g(z_{k})(u_{k}-z_{k})+g(z_{k})\nleq 0$).
% \end{proof}

\subsubsection{Gradient-Free Cut}
We replace the Euclidean projection cut with its gradient-free version. Instead of using $\nabla g$, we use the Euclidean projection $z_k$ and define the gradient-free cut as follows: 

% Next, we show that the gradient-free cut also supporting hyperplanes and exclude violated points. The gradient-free cuts are as follows.

\begin{align}\label{eq:eudcut2}
(u_{k}-z_{k})^\top u \leq (u_{k}-z_{k})^\top z_{k}.
\end{align}
Inequality \eqref{eq:eudcut2} is valid for $\mathcal{U}$ because of the variational characteristic of projection \cite{proj_prop}. Because $z_{k}\in\partial\mathcal{U}$, \eqref{eq:eudcut2} supports $\mathcal{U}$. The exclusion of $u_{k}$ can also be seen by plugging it in because $u_{k}\neq z_{k}$.

\begin{Exp}
Here we consider the following NRO problem to demonstrate the difference between the cutting planes:
    \begin{align}
    \min_{x_1,x_2}\quad &-x_1-x_2\label{eq:rob0308}\\
    \text{s.t.}\quad &x_1^2u_1+x_2^2u_2\leq 6,\ \forall (u_1,u_2)^\top\in\mathcal{U},\nonumber
\end{align}
where $\mathcal{U}=\{u_1^2 + u_2^2\leq 1, u_1\geq 0, u_2\geq 0\}$. The solution of \eqref{eq:rob0308} is $x^*=(\sqrt{3\sqrt{2}},\sqrt{3\sqrt{2}})$ with  optimal multiplier $\lambda^*=\sqrt{3\sqrt{2}}/6$ and $u^*=(\sqrt{2}/2,\sqrt{2}/2)$.

     \begin{figure}[htb]
    \centering
    \begin{subfigure}{0.3\textwidth}
        \includegraphics[width=\textwidth]{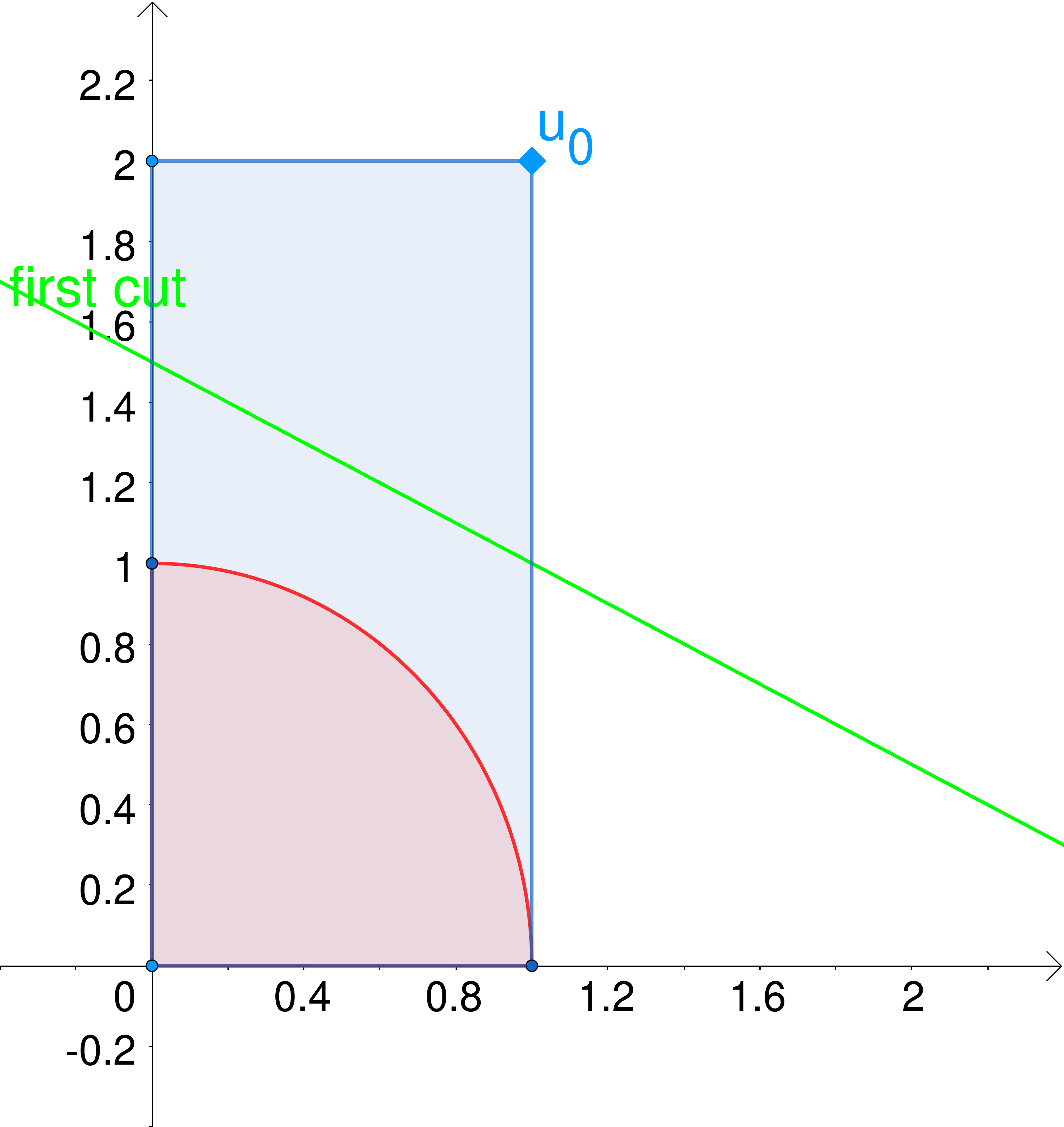}
        % \caption{}
    \end{subfigure}\quad
    \begin{subfigure}{0.3\textwidth}
        \includegraphics[width=\textwidth]{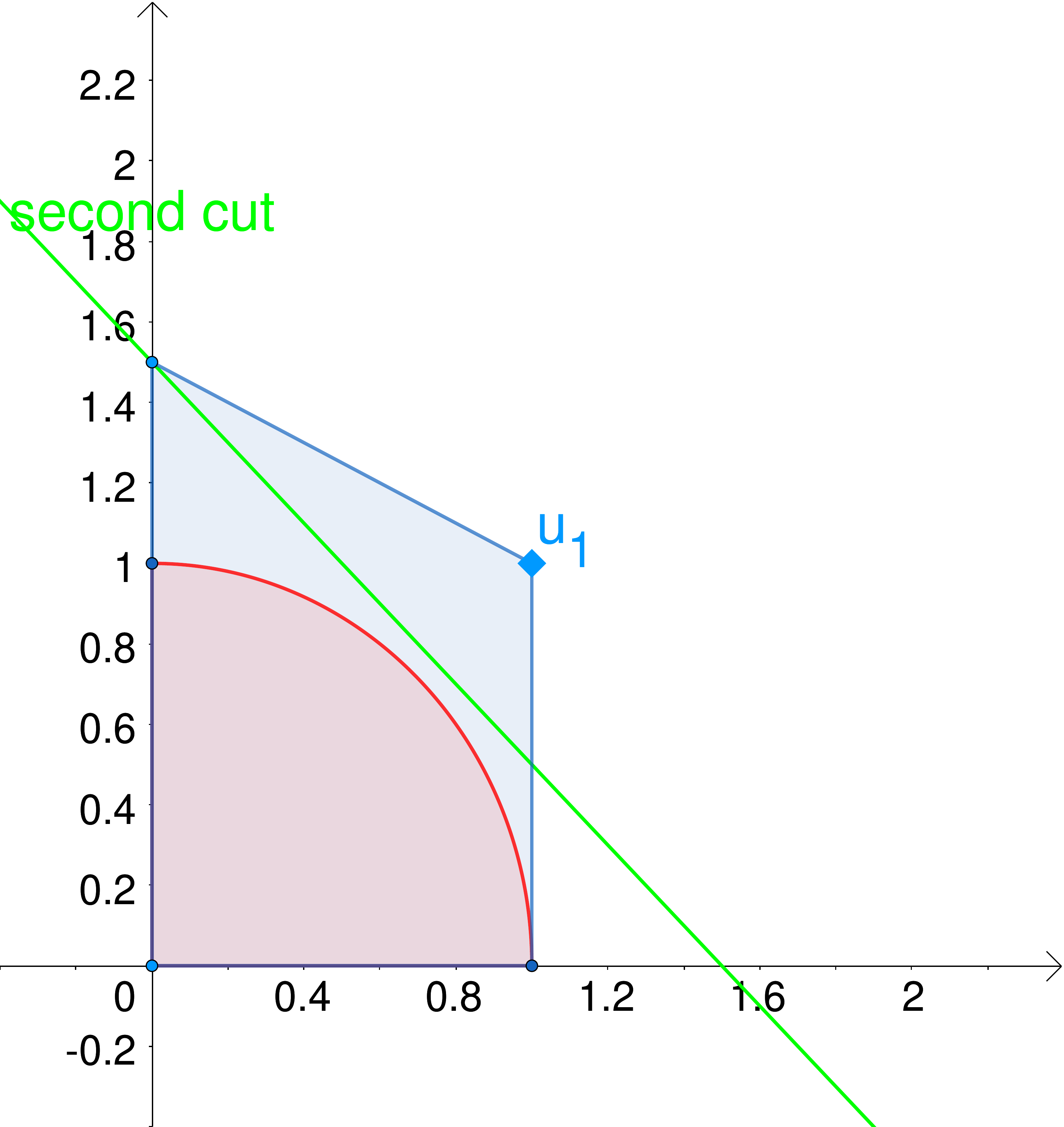}
        % \caption{Superset $\mathcal{S}_1$: $\epsilon_1=0.478$.}
    \end{subfigure}\quad
    \begin{subfigure}{0.3\textwidth}
        \includegraphics[width=\textwidth]{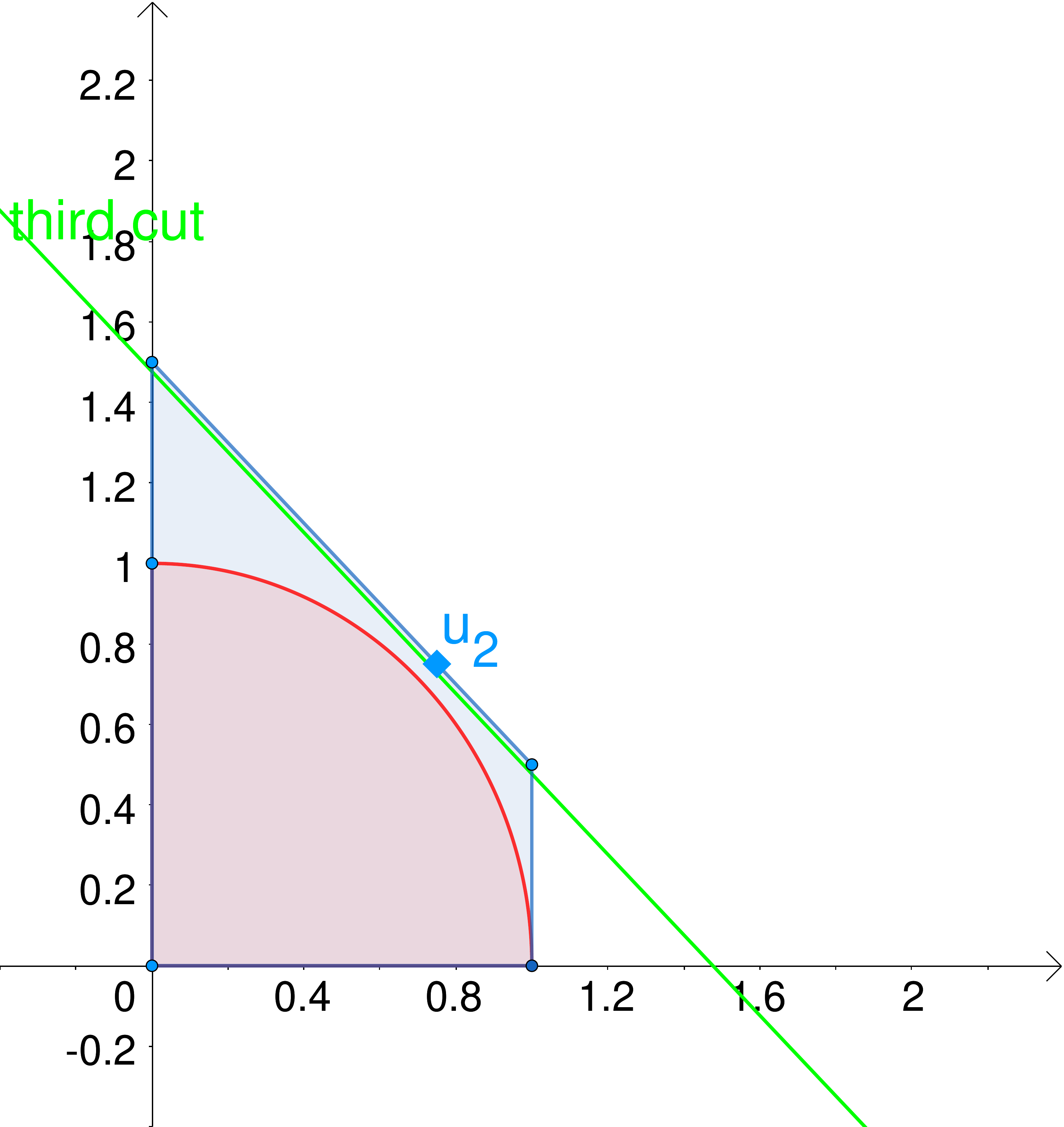}
        % \caption{Superset $\mathcal{S}_1$: $\epsilon_1=0.478$.}
    \end{subfigure}\vspace{0.1cm}\\
    \begin{subfigure}{0.3\textwidth}
        \includegraphics[width=\textwidth]{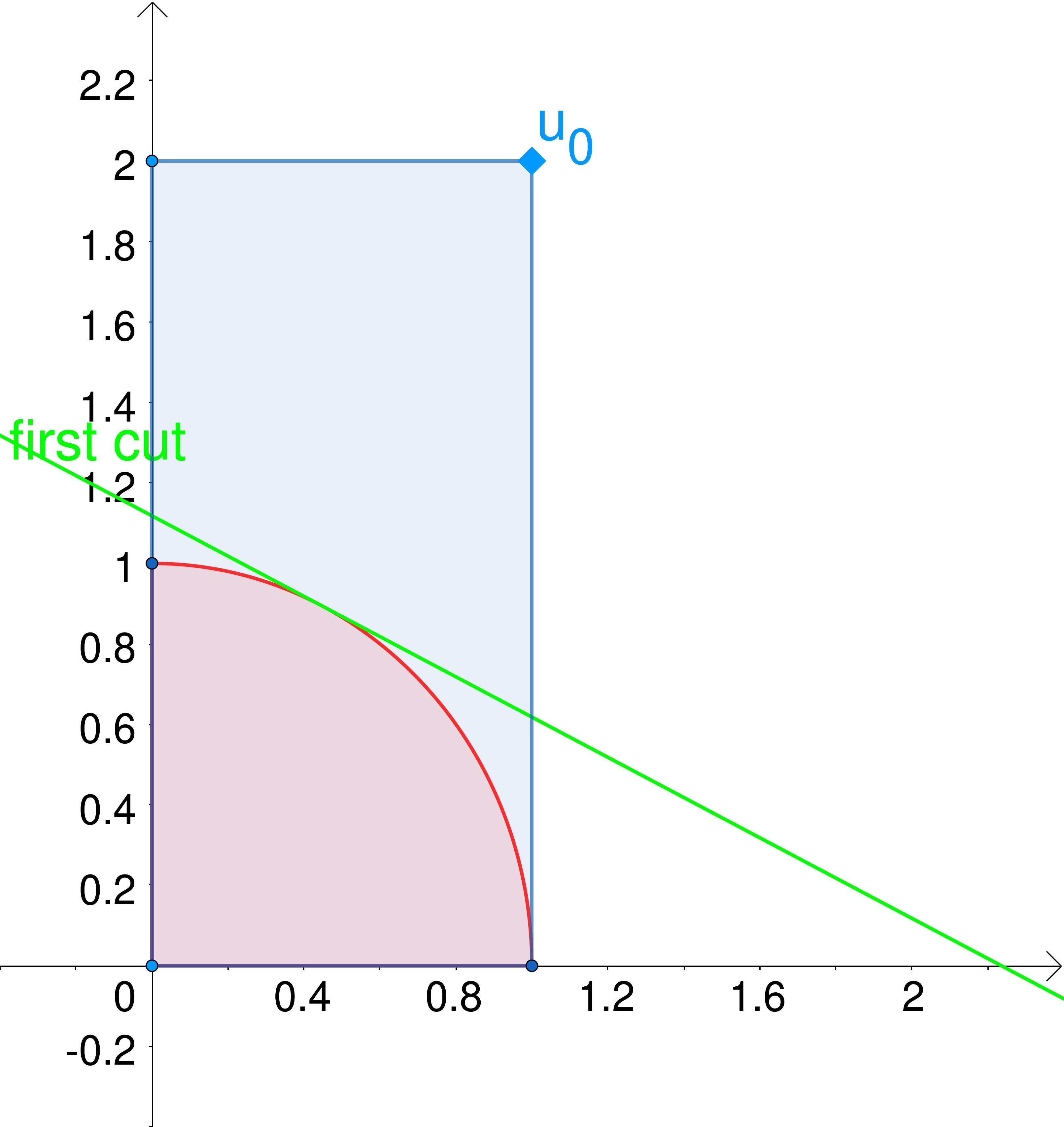}
        \caption{$\mathcal{S}_0$ and first cut.}
    \end{subfigure}\quad
    \begin{subfigure}{0.3\textwidth}
        \includegraphics[width=\textwidth]{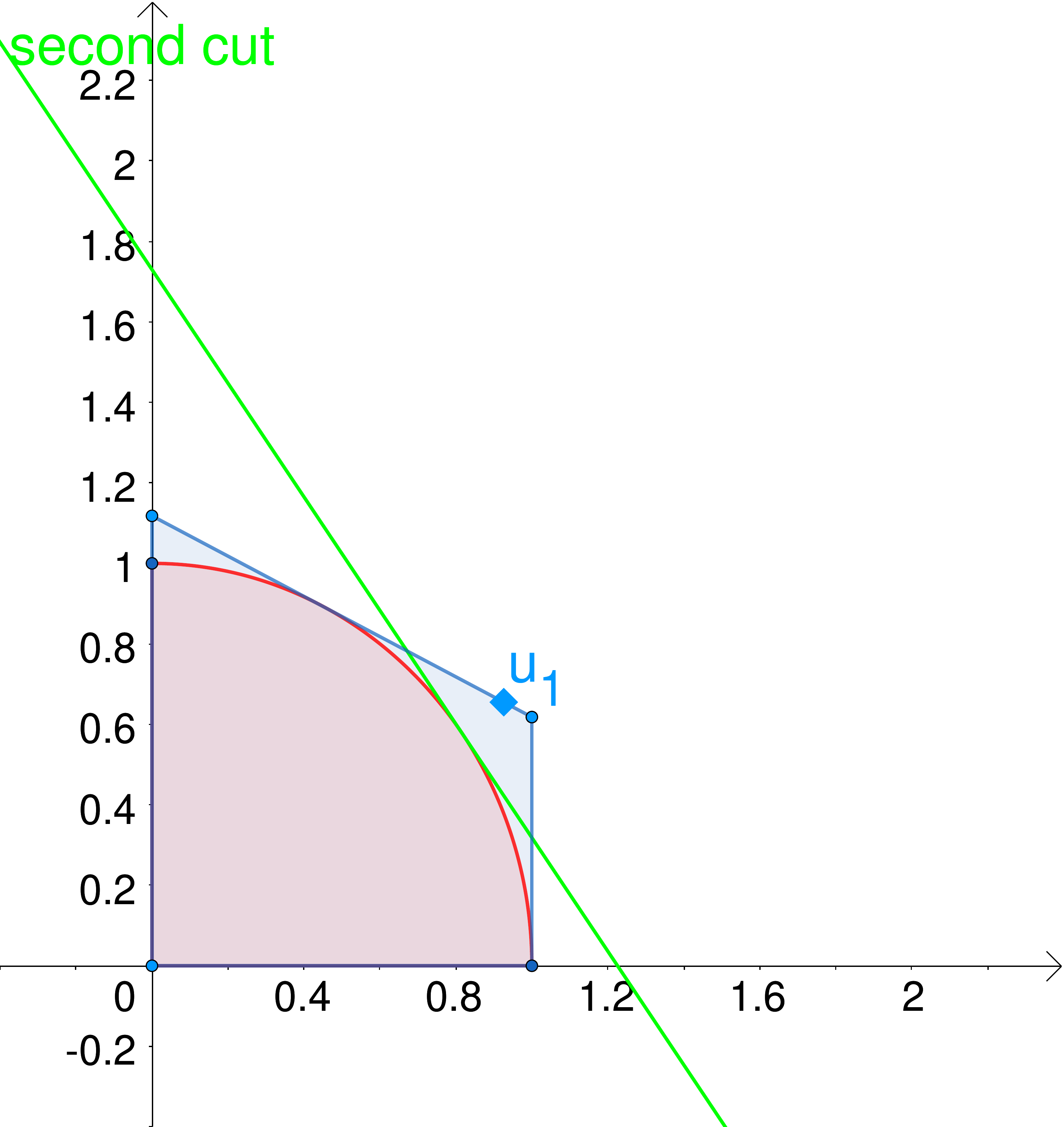}
        \caption{$\mathcal{S}_1$ and second cut.}
    \end{subfigure}\quad
    \begin{subfigure}{0.3\textwidth}
        \includegraphics[width=\textwidth]{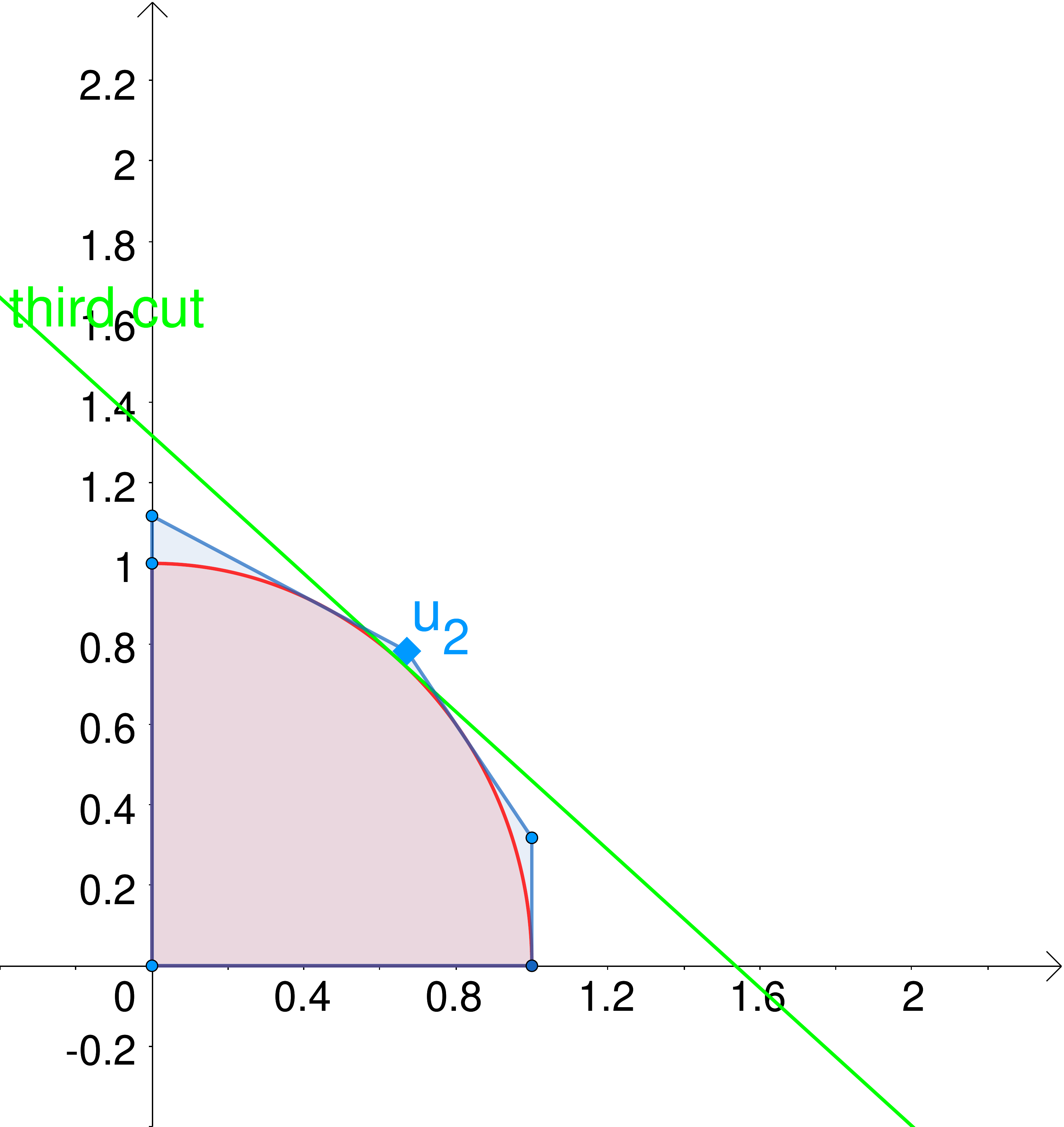}
        \caption{$\mathcal{S}_2$ and third cut.}
    \end{subfigure}
    \caption{Illustration of the superset update in Algorithm~\ref{alg:sw} using Kelley's cutting plane (top) and Euclidean projection cut (bottom).}\label{fig:illu0}
    \end{figure}
    
    In Figure \ref{fig:illu0} we illustrate the superset update in Algorithm~\ref{alg:sw_sketch} with all proposed types of cutting planes for the first three iterations (i.e., $k=0,1,2$). The algorithm starts with initial superset $\mathcal{S}_0=\{0 \leq u_1 \leq 1,\ 0\leq u_2\leq 2\}$ (rectangular). In this example, gradient-free cuts coincide with the Euclidean projection cuts and result in the same algorithm progress because of the simplicity of the uncertainty set. We observe that the Euclidean projection cut and gradient-free cut are more effective than Kelley's cutting plane in terms of improving the superset to better approximate the uncertainty set. The second cut and third cut from Kelley's cutting method are close to each other and can be less effective. The detailed iteration steps are summarized in Table~\ref{tab:kelley_eud}.

\end{Exp}  

% \begin{landscape}
% \thispagestyle{empty}
% \vspace{5cm}
\begin{table}[ht]
\centering
     \caption{Solution pair $(x_k, u_k, \lambda_k)$ and cutting planes obtained in the first three iterations of Algorithm~\ref{alg:sw_sketch} (i.e., $k = 0,1,2$) where $a_1 = \frac{4\sqrt{6}+\sqrt{12}-\sqrt{5}-2\sqrt{10}}{7}$ and $a_2 = \frac{(\sqrt{10}-\sqrt{3})(1+2\sqrt{2})}{7}$.}
  \centering
    \begin{subtable}[h]{1\textwidth}
    \centering
       \caption{Kelley's Cutting Plane.}
\begin{tabular}{c|c|c|c|c}
\hline
\begin{tabular}[c]{@{}c@{}}Iteration\\ $k$\end{tabular} & $x_k$ & $u_k$ & $\lambda_k$ & $k$th Cutting Plane \\ \hline\hline
$0$                                                     & $(2,1)$   & $(1,2)$   &   $1/4$      & $u_1 + 2u_2\leq 3$                                                            \\ \hline
$1$                                                     & $(\sqrt{3},\sqrt{3})$    &  $(1,1)$  & $\sqrt{3}/6$       &  $u_1 + u_2\leq 1.5$                                                          \\ \hline
$2$                                                     &  $(2,2)$  &  $(0.75,0.75)$  & $1/3$        &  $u_1+u_2\leq 443/300$                                                          \\ \hline
\end{tabular}
       \label{tab:kelley}
    \end{subtable}
    \newline
    \vspace{0.2cm}
    \newline
    \begin{subtable}[h]{1\textwidth}
     \centering
    \caption{Euclidean Projection Cut.}
         \centering
\begin{tabular}{c|c|c}
\hline
\begin{tabular}[c]{@{}c@{}}Iteration\\ $k$\end{tabular} & $x_k$ & $u_k$ \\ \hline\hline
0                                                     & $(2,1)$   &  $(1,2)$                                 \\[3pt] \hline
1                                                     & $\left(\frac{\sqrt{30\sqrt{5}}}{5},\frac{2\sqrt{15\sqrt{5}}}{5}\right)$   & $\left(\sqrt{5}(\sqrt{2}-1),\frac{\sqrt{10}(\sqrt{2}-1)}{2}\right)$                                                         \\[5pt] \hline
2                                                     &  $\left(\sqrt{\frac{6a_2}{a_1a_2 + a_1^2}},\sqrt{\frac{6a_1}{a_1a_2 + a_2^2}}\right)$ &  $(a_1,a_2)$                                                        \\[6pt] \hline
\end{tabular}
\newline
\vspace{0.2cm}
\newline
\begin{tabular}{c|c|c}
\hline
\begin{tabular}[c]{@{}c@{}}Iteration\\ $k$\end{tabular} & $\lambda_k$ & $k$th Cutting Plane \\ \hline\hline
0                                                      & $1/4$        &   $\sqrt{5}u_1 + 2\sqrt{5}u_2\leq 5$                                                         \\[3pt] \hline
1                                                       &   $\frac{\sqrt{6}(\sqrt{2}+1)}{12}$      &   $\sqrt{6}u_1+\sqrt{3}u_2\leq3$                                                         \\[5pt] \hline
2                                                     &   $\frac{\sqrt{a_1+a_2}}{2\sqrt{6a_1a_2}}$      &   $\sqrt{\frac{a_1^2}{a_1^2 + a_2^2}}u_1+\sqrt{\frac{a_2^2}{a_1^2 + a_2^2}}u_2\leq 1$                                                         \\[6pt] \hline
\end{tabular}
        \label{tab:eud}
     \end{subtable}
     \label{tab:kelley_eud}
\end{table}
% \end{landscape}
 % In addition, we measure the solution error (i.e., $\epsilon_k=\|x_k-x^*\|_2$ where $x^*$ denotes the true optimal solution). 

\subsection{Termination Criteria}\label{sec:termination}

% To terminate the superset algorithm, we extend the error measure \eqref{eq:robkkt_err} to the reformulation \eqref{eq:nonlin}. For brevity, we drop the iteration subscript $k$. 

We consider the NRO subproblem \eqref{eq:supsub} using $\mathcal{S}$ and denote its solution pair by $(\bar{x},\bar{\lambda},\bar{u})$ (from its reformulation \eqref{eq:nonlin}). To evaluate the solution quality to \eqref{eq:original}, we use \eqref{eq:robkkt_err} as follows: 
\begin{align}\label{eq:robkkt_err_sup}
    errm(\bar{x},\bar{\lambda},\{\bar{u}\})=\|\bar{u}-\text{proj}_{\mathcal{U}}(\bar{u})\|_2.
\end{align}
Because $\bar{u}\in\mathcal{S}$ and $\mathcal{S}\supseteq\mathcal{U}$, \eqref{eq:robkkt_err_sup} can be nonzero. All the remaining components in \eqref{eq:robkkt_err} are zero because of the proof of Theorem~\ref{thm:main}. For brevity, we will use a short notation $errm(x)$ in the future discussion of the superset algorithm.

\paragraph{Comparison with Polak's Subproblem}
As comparison, we consider the sample-based finite nonlinear problem \eqref{eq:Polak_fin} using $U_{k}$ and denote its solution as $\hat{x}$. Define a finite set $\hat{A} := \{u\in U_{k} \ \lvert\ u^\top h(\hat{x}) = b(\hat{x})\}$, and let $\hat{\lambda}_u$ be the corresponding optimal multipliers of \eqref{eq:Polak_fin} for all $u\in\hat{A}$. We use \eqref{eq:robkkt_err} to evaluate the solution quality of $\hat{x}$ to \eqref{eq:original} as follows:

\begin{align}\label{eq:robkkt_err_pol}
    errm(\hat{x},\hat{\lambda}_u,\hat{A})=\left[\max_{u\in\mathcal{U}}u^\top h(\hat{x}) - b(\hat{x})\right]^+.
\end{align}
Because $\hat{x}$ can  guarantee feasibility only for the given sample set $U_k$ but not for the original uncertainty set $\mathcal{U}$, \eqref{eq:robkkt_err_pol} can be nonzero. All the remaining components \eqref{eq:7a}, \eqref{eq:7c}, and \eqref{eq:7d} are zero because of the KKT condition of \eqref{eq:Polak_fin}.

From the comparison above, we observe that the solutions of the polytopic superset NRO and sample-based subproblem of Polak's algorithm achieve different KKT guarantees. The iterates from  Polak's algorithm violate primal feasibility. On the other hand, the iterates from the superset algorithm violate the activity check but are guaranteed to satisfy primal feasibility. 

\subsection{Feasibility Restoration}\label{rmk:feas}
If \eqref{eq:supsub} is infeasible for $k=0$, there are two possibilities:  either the original problem \eqref{eq:original} is infeasible or  the initial $\mathcal{S}_0$ is too conservative. To detect whether \eqref{eq:original} is infeasible or restore the feasibility of \eqref{eq:supsub}, we define the following phase-I problem to \eqref{eq:original}:
    \begin{align}\label{eq:phase1}
        \min_{x,p\geq 0}\  p \quad \mbox{s.t. } u^\top h(x)\leq b(x)+p,\ \forall u\in\mathcal{U}.
\end{align}

One can easily  see that \eqref{eq:phase1} is always feasible. 
% whether \eqref{eq:original} is feasible or not with the following proposition. 

\begin{prop}
    \eqref{eq:phase1} has optimal objective value $p^*=0$ if and only if \eqref{eq:original} is feasible.
\end{prop}

The proof can be obtained by direct checking. To solve \eqref{eq:phase1}, we will use the superset algorithm starting with the initial box superset $\hat{\mathcal{S}}$. We define the NRO subproblem to \eqref{eq:phase1} with polytopic superset $\mathcal{S}_k$ at iteration $k$ as 
\begin{align}\label{eq:phase1sup}
        \min_{x,p\geq 0}\  p\quad \mbox{s.t. } u^\top h(x)\leq b(x)+p,\ \forall u\in\mathcal{S}_{k}.
\end{align}
Similar to \eqref{eq:phase1} being a phase-I problem to \eqref{eq:original}, \eqref{eq:phase1sup} is a phase-I problem to \eqref{eq:supsub} and can be solved with its reformulation from Proposition~\ref{prop:robref}. Based on \eqref{eq:phase1sup}, we design the following feasibility restoration algorithm.

\begin{algorithm}[H]
\begin{algorithmic}[1]

\State {\bf initialization:} Set $k\gets0$, and initialize superset $\mathcal{S}_0 \gets \hat{\mathcal{S}}$ and some $\epsilon>0$.\;

\Repeat

\State {\bf superset subproblem:} Let $({x}_{k}, p_k)$ solve \eqref{eq:phase1sup}.

    \If {$ p_{k} = 0$}
        % \State Construct the superset $\tilde{\mathcal{S}}$ following Proposition~\ref{prop:fr}.
        \State {\bf return} ${x}_{k}$, $p_{k}$, $\mathcal{S}_{k}$. \Comment{$x_k$ is feasible in \eqref{eq:original}.} 
    \EndIf

\State {\bf update superset:} Update $\mathcal{S}_{k}$ with the following cutting planes:
\State \hskip3em $\bullet$ Kelley's cutting plane: $\nabla g(u_{k})^\top u \leq \nabla g(u_{k})^\top u_{k}-g(u_{k})$, or
\State \hskip3em $\bullet$ Euclidean Projection cut: $\nabla g(z_{k})^\top u \leq \nabla g(z_{k})^\top z_{k}-g(z_{k})$, or
\State \hskip3em $\bullet$ Gradient-free cut: $(u_{k}-z_{k})^\top u \leq (u_{k}-z_{k})^\top z_{k}.$

\State Set $k\gets k+1$.

\Until{$errm({x}_{k-1})\leq \epsilon$.}
    
\State {\bf return} ${x}_{k-1}$, $p_{k-1}$, $\mathcal{S}_{k-1}$. \Comment{\eqref{eq:original} is infeasible with $p_{k-1}>0$.} 
    
\end{algorithmic} 
\caption{Feasibility Restoration: $(x, p, \mathcal{S})\leftarrow\text{FR}(\hat{\mathcal{S}}, \epsilon). $}\label{alg:feas}
\end{algorithm}

% Given the feasible point ${x}_k$, we can construct a new superset $\tilde{\mathcal{S}}$ following Proposition~\ref{prop:fr}. The returned superset $\tilde{\mathcal{S}}\cap\mathcal{S}_k$ will then be used in the main loop of Algorithm~\ref{alg:sw_sketch} with a feasibility guarantee on \eqref{eq:supsub}. We return intersection $\tilde{\mathcal{S}}\cap\mathcal{S}_k$ rather than $\tilde{\mathcal{S}}$ alone to reuse the valid cutting planes, derived from the feasibility restoration step, in the main loop of Algorithm~\ref{alg:sw_sketch}. 
    
Here the error measure $errm$ in line 12 follows the same definition to \eqref{eq:robkkt_err_sup}. Algorithm~\ref{alg:feas} has two termination criteria. When $p_{k} = 0$ in line 4, we obtain ${x}_k$, which  is feasible in \eqref{eq:original}. In this case, superset $\mathcal{S}_k$ is returned and will be used in the main loop of Algorithm~\ref{alg:sw_sketch}. Otherwise, in line 13 we terminate with an $\epsilon$-stationary point $x_{k-1}$ to the phase-I problem \eqref{eq:phase1} with a certificate of infeasibility $p_{k-1}>0$. In Section~\ref{sec:conv} we show that the Algorithm~\ref{alg:feas} results in the convergence to stationarity. 

% \begin{prop}\label{prop:fr}
%      Given solution $x_k$ at iteration $k$, define sets $\tilde{\mathcal{S}} = \{u:\ B^*u\leq d^*\}$ where $B^*= h^{\top}({x}_k)$ and $d^* = \max_{u\in\mathcal{U}}u^\top h(x_k)$. If $ p_{k} = 0$ (i.e., $x_k$ is feasible in \eqref{eq:original}), then the following properties hold:
%     \begin{enumerate}
%         \item {\bf Superset:} $\mathcal{U}\subseteq\tilde{\mathcal{S}}$. 
%         \item {\bf Feasibility:} ${x}_k$ is feasible in the following superset problem:
%         \begin{align}\label{eq:supsub2}
%         \min_{x}\ f(x)\quad\mbox{s.t. } u^\top h(x)\leq b(x),\ \forall u\in\tilde{\mathcal{S}}\cap\mathcal{S}_k.
% \end{align}
%     \end{enumerate}
% \end{prop}

% The proof can be easily obtained by directly checking. 

\subsection{Full Algorithm}\label{sec:full_alg}

The complete steps of the superset algorithm are described in Algorithm~\ref{alg:sw}. First, we recall the Euclidean projection $z_k = \text{proj}_\mathcal{U}(u_k)$. 

\begin{algorithm}[H]
\begin{algorithmic}

\State {\bf initialization:} Set $k\gets 0$; initialize a box superset $\hat{\mathcal{S}}\supseteq\mathcal{U}$ and some $\epsilon>0$.

\State {\bf feasibility restoration:} $(x_0,p_0,\mathcal{S}_0) \leftarrow\text{FR}(\hat{\mathcal{S}}, \epsilon)$. 

\If{$p_0>0$}
\State {\bf return} $x_0$ and the infeasibility of \eqref{eq:original}, $p_0$. %if $p_0>0$.
\EndIf

\Repeat
    
\State {\bf superset subproblem:} Let $x_{k}$ be the solution of \eqref{eq:nonlin}.

\State {\bf update supersets:} Update $\mathcal{S}_{k}$ with cutting planes: 

\State \hskip3em $\bullet$ Kelley's cutting plane: $\nabla g(u_{k})^\top u \leq \nabla g(u_{k})^\top u_{k}-g(u_{k})$, or
\State \hskip3em $\bullet$ Euclidean Projection cut: $\nabla g(z_{k})^\top u \leq \nabla g(z_{k})^\top z_{k}-g(z_{k}),$ or
\State \hskip3em $\bullet$ Gradient-free cut: $(u_{k}-z_{k})^\top u \leq (u_{k}-z_{k})^\top z_{k}.$

\State Set $k\gets k+1$.
    
\Until{$errm(x_{k-1})\leq\epsilon.$}    
    
\State {\bf return} $x_{k-1}$.
 
\end{algorithmic}
\caption{Superset Algorithm}\label{alg:sw}
\end{algorithm}
    
If \eqref{eq:original} is infeasible or if the initial box superset $\hat{\mathcal{S}}$ is overly conservative, then \eqref{eq:supsub} can be infeasible. Here we use the feasibility restoration step (see Section~\ref{rmk:feas}) to either restore the feasibility by constructing a new superset $\mathcal{S}_0$ for the NRO subproblem \eqref{eq:supsub} or return a stationary point $x_0$ to a phase-I problem \eqref{eq:phase1} with an infeasibility certificate $p_0>0$.  

After exiting feasibility restoration with $p_0 = 0$, we iteratively solve the NRO subproblem \eqref{eq:supsub} and update the superset $\mathcal{S}_k$ with cutting planes. We define the feasible set of \eqref{eq:supsub} at iteration $k$ as $\mathcal{F}_k$. As we keep adding cutting planes, we conclude that $\mathcal{F}_k\subseteq \mathcal{F}_{k+1}\subseteq \mathcal{F}$ for all $k$ (i.e., the feasible set of \eqref{eq:supsub} gets larger) and the sequence of the resulting supersets $\{\mathcal{S}_k\}$ is all compact sets. We  also achieve a monotonic convergent upper bound on the optimal objective value of \eqref{eq:original} (i.e., $ f(x_k)\geq f(x_{k+1})\geq f^*$, where $f^*$ is the optimal objective value of \eqref{eq:original}). In Section~\ref{sec:proj_cuts} we will show that Algorithm~\ref{alg:sw} generates solution iterates that converge to the solution of \eqref{eq:original} with the proposed cutting plane types.

In addition to the upper bound, a lower bound on the optimal objective value can  be computed by solving a problem like \eqref{eq:Polak_fin}, whose sample sets are collected from the Euclidean projection $z_k$ in the error measure $errm(x_k)$ \eqref{eq:robkkt_err_sup}. Alternatively, we can combine the superset algorithm with any effective outer approximation methods to obtain faster lower bound convergence (e.g., \cite{dis1,tut4, outer1}). This combined algorithm will have a decreasing objective gap to terminate with.

Compared with Polak's algorithm, the superset algorithm is capable of providing both lower and upper bounds to the optimal objective value of \eqref{eq:original}, whereas  Polak's algorithm provides only a lower bound. In terms of feasibility, the superset algorithm generates iterates that are feasible in \eqref{eq:original} whereas the iterates of Polak's algorithm are generally infeasible in \eqref{eq:original} before convergence. This allows us to define a superset algorithm with early terminations in practical applications if needed. In terms of assumption requirements, the superset algorithm additionally requires convexity of the uncertainty set $\mathcal{U}$, so that the cutting plane method can be used to update supersets over iteration.

\section{Convergence Analysis}\label{sec:conv}
We prove the convergence for the sequence of the NRO subproblem solutions generated by Algorithm~\ref{alg:sw}. We develop a unified convergence proof for Algorithm~\ref{alg:sw} with each of the proposed cutting plane types in Section~\ref{sec:proj_cuts}. This implies that a hybrid cutting plane strategy also converges. 

% The proof is similar to a convergence proof of a cutting plane algorithm for differentiable convex programs \cite{Kelley}.

\begin{prop}\label{lem:kelley_conv1}
In Algorithm \ref{alg:sw} with Kelley's cutting plane, \eqref{eq:kelleycut}, Euclidean projection cut \eqref{eq:eudcut}, or gradient-free cut \eqref{eq:eudcut2}, there exists a convergent subsequence indexed by $\{k_p\}$ and $u^*\in\mathcal{U}$ such that $\{u_{k_p}\}\rightarrow u^*$.
\end{prop}

\begin{proof}
From Algorithm~\ref{alg:sw}, we obtain a sequence of solutions $\{x_k\}$ and uncertainty realizations $\{u_{k}\}$. We define the corresponding Euclidean projection sequence $\{z_{k}\}$, where $z_k = \text{proj}_{\mathcal{U}}(u_k)$ for all $k$. We claim that there exists a convergent subsequence, indexed by $\{k_p\}$, and $u^*\in\mathcal{U}$ such that $\{u_{{k_p}}\}\rightarrow u^*$. 

Next, we consider each cutting plane in turn and prove the claim by contradiction. If the convergence does not occur, then there exists $r>0$ (independent of $k$) such that for any given $k$ and $t$ and $0\leq t \leq k-1$, we have the following:
\begin{enumerate}
    \item Kelley's cutting plane \eqref{eq:kelleycut}: \begin{align*}
    r&\leq \max_{j\in J} g_{j}(u_{t})\leq  \max_{j\in J} \nabla g_{j}(u_{t})^\top(u_{t}-u_{k})\leq L\|u_{t}-u_{k}\|_2,
\end{align*}
where the second inequality comes from the fact that $u_k$ satisfies  Kelley's cutting plane at iteration $t$ and the following inequality comes from the Lipschitz continuity of $\nabla g_j$. Then, it follows that for every subsequence (i.e., $k_l<k_m$), we have
\begin{align}\label{eq:vio1}
    \|u_{{k_l}}-u_{k_m}\|_2\geq r/L.
\end{align}

\item Euclidean projection cut \eqref{eq:eudcut}: 
\begin{subequations}
\begin{align}
    r\leq \|u_{t}-z_{t}\|_2^2&= (u_{t}-u_{k})^\top(u_{t}-z_{t})+(u_{k}-z_{t})^\top(u_{t}-z_{t})\\
    &\leq (u_{t}-u_{k})^\top(u_{t}-z_{t})\label{eq:eud_sub1_gf}\\
    &\leq \|u_{t}-u_{k}\|_2\|u_{t}-z_{t}\|_2\\
    &\leq 2R\|u_{t}-u_{k}\|_2,\label{eq:eud_sub2_gf}
\end{align}
\end{subequations}
where \eqref{eq:eud_sub2_gf} comes from the fact that $u_{t},z_{t}\in\mathcal{S}_{0}$. By construction, $\mathcal{S}_{0}$ is a compact set and hence bounded with radius $R$.
Next, we show \eqref{eq:eud_sub1_gf} by proving $(u_{t}-z_{t})^\top(u_{k}-z_{t})\leq 0$. For all $j\in \hat{A}_{t}=\{j\in J:  g_{j}(z_{t}) = 0\}$ , we have
\begin{align}\label{eq:conv2}
    g_{j}(z_{t}) = 0\ \Rightarrow\ \nabla g_{j}(z_{t})^\top(u_{k}-z_{t})\leq 0, 
\end{align}
because $u_{k}$ satisfies the Euclidean projection cut \eqref{eq:eudcut} at iteration $t$. Next, based on the KKT condition of \eqref{eq:eud}, we have
\begin{align}\label{eq:optcut2}
        2(z_{t}-u_{t})+\sum_{j\in\hat{A}_t}\lambda_{jt}\nabla g_{j}(z_{t})=0,
    \end{align}
where $\lambda_{jt}\geq 0$ are the optimal multipliers for all $j\in\hat{A}_t$. Then, from \eqref{eq:conv2} and \eqref{eq:optcut2}, we get
\begin{align*}
   2(u_{t}-z_{t})^\top(u_{k}-z_{t})= \sum_{j\in\hat{A}_{t}}\lambda_{jt}\nabla g_{j}(z_{t})^\top(u_{k}-z_{t})\leq 0.
\end{align*}

Next, from \eqref{eq:eud_sub2_gf}, it follows that for every subsequence (i.e., $k_l<k_m$), we have
\begin{align}\label{eq:vio2}
    \|u_{{k_l}}-u_{{k_m}}\|_2\geq r/2R.
\end{align}

% we multiply the right-hand side of \eqref{eq:conv2} with the optimal multipliers $\lambda_{jt}$
% (i.e., $\hat{A}$ at iteration $t$ with $\hat{A}$ defined in the proof of Proposition~\ref{prop:active cut})

% i.e., optimal multipliers $\lambda_j$ satisfies at iteration $t$ with $\lambda_j$ defined in Proposition~\ref{prop:active cut}

\item Gradient-free cut \eqref{eq:eudcut2}: \begin{subequations}
\begin{align}
    r\leq \|u_{t}-z_{t}\|_2^2&= (u_{t}-u_{k})^\top(u_{t}-z_{t})+(u_{k}-z_{t})^\top(u_{t}-z_{t})\\
    &\leq (u_{t}-u_{k})^\top(u_{t}-z_{t})\label{eq:sub1_gf}\\
    &\leq \|u_{t}-u_{k}\|_2\|u_{t}-z_{t}\|_2\\
    &\leq 2R\|u_{t}-u_{k}\|_2,\label{eq:sub2_gf}
\end{align}
\end{subequations}
where \eqref{eq:sub1_gf} comes from the fact that $u_{k}$ satisfies the cutting plane generated at iteration $t$, that is, 
\begin{align*}
    (u_{k}-z_{t})^\top(u_{t}-z_{t})\leq 0,
\end{align*}
and \eqref{eq:sub2_gf} comes from the fact that $u_{t},z_{t}\in\mathcal{S}_{0}$ and $\mathcal{S}_{0}$ is compact and hence bounded by construction. Then, it follows that for every subsequence (i.e., $k_l<k_m$), we have
\begin{align}\label{eq:vio3}
    \|u_{{k_l}}-u_{{k_m}}\|_2\geq r/2R.
\end{align}
\end{enumerate}

Inequalities \eqref{eq:vio1}, \eqref{eq:vio2}, and \eqref{eq:vio3} imply that $\{u_{k}\}$ does not contain a Cauchy subsequence, which contradicts  the fact that $\{u_{k}\} \in \mathcal{S}_{0}$ because the compactness of $\mathcal{S}_{0}$ implies that any sequence in $\mathcal{S}_0$ contains a Cauchy subsequence. Therefore, $\{u_{k}\}$ contains a convergent subsequence with indices $\{k_p\}$ that converges to a point $u^*\in \mathcal{U}$. 
\end{proof}

\begin{thm}\label{thm:conv_kelly}
With the convergent subsequence (with indices $\{k_p\}$) generated by Algorithm~\ref{alg:sw}, we obtain the following three conclusions: 
\begin{enumerate}
    \item The sequence of $\{\|z_{{k_p}}-u_{{k_p}}\|_2\}$ converges to $0$.
    \item The sequence of $\{z_{{k_p}}\}$ converges to $u^*$. 
    \item There exist $x^*\in\mathcal{F}$ and a subsequence of $\{x_{k_p}\}$ with indices $\{k_l\}$ such that $\{x_{k_l}\}$ converges to $x^*$ and $\{errm(x_{k_l})\}$ converges to 0. 
\end{enumerate}
\end{thm}
\begin{proof}
We have $\{\|z_{{k_p}}-u_{{k_p}}\|_2\}\rightarrow 0$ because 
    \begin{align*}
        0\leq \|z_{{k_p}}-u_{{k_p}}\|_2 \leq \|u_{{k_p}}-u^*\|_2. 
    \end{align*}
The inequality holds because $z_{k}$ is the Euclidean projection of $u_{k}$ onto the uncertainty set $\mathcal{U}$, and we have $\{z_{{k_p}}\}\rightarrow u^*$ because
\begin{align*}
     0 &\leq \|z_{{k_p}}-u^*\|_2 \leq \|z_{{k_p}}-u_{{k_p}}\|_2 + \|u_{{k_p}}-u^*\|_2.
\end{align*}

Because $\mathcal{F}$ is compact and $x_{k_p}\in\mathcal{F}$ for all $k_p$, there exists a subsequence with indices $\{k_l\}$ and $x^*\in\mathcal{F}$ such that $\{x_{k_l}\}\rightarrow x^*$. Further, we have $\{errm(x_{k_l})\}\rightarrow 0$ because of conclusion 1. 
\end{proof}

Following the third conclusion in Theorem~\ref{thm:conv_kelly}, we conclude that with any cutting planes \eqref{eq:kelleycut}, \eqref{eq:eudcut}, or \eqref{eq:eudcut2}, Algorithm~\ref{alg:sw} converges to a stationary solution of \eqref{eq:original}.

\begin{rmk}
    Comparing the convergence analysis of Polak's algorithm \cite{polak_boyd} and our superset algorithm, we see that the main difference is that Polak's algorithm uses an assumption that problem \eqref{eq:feas_err} is solved to global optimality. In our superset algorithm, because we aim to improve the supersets by removing the violated points, we assume that $\mathcal{U}$ is a convex set to take advantage of the cutting plane methods. Note that convexity is a sufficient condition for global optimality, which means that the two sets of assumptions are not significantly different. 
\end{rmk}

\section{Numerical Results}\label{sec:numerical}

We compare the computational performance of the superset algorithm and Polak's algorithm using the following applications: portfolio optimization and production cost minimization. 

\subsection{Simulation Setup}

We implement Polak's and the superset algorithms using Julia 1.6.5 and run the numerical studies on a Linux system with a single-thread, 2.20 GHz Intel(R) Xeon(R) CPU E7-8890 and 16 GB RAM. We use the filterSQP solver \cite{filtersqp} for all the smooth nonlinear optimization problems and Mosek 9.3 \cite{mosek} for the semi-definite programming problem in portfolio optimization. In the superset algorithm, we use the Euclidean projection cut and the gradient-free cut. For termination tolerance, we use $\epsilon=10^{-5}$.

\subsection{Robust Asset Allocation} \label{sec:portfolio}

We consider the following multiperiod asset allocation problem with transaction costs \cite{port1,port2,port_horizon1, port_horizon2}. We define $n$ to be the number of assets and the decision horizon as $t\in\{1,...,T\}$. Coefficient $U$ denotes the time-invariant linear transaction cost that is independent of the asset type. For each period $t$,  $x_t\in\mathbb{R}^n$ represents the allocation decision for each asset, and $\lambda_t>0$ is a predefined risk aversion parameter. Parameters $\mu_t\in\mathbb{R}^n$ and $Q_t\in\mathbb{R}^{n\times n}$ represent the uncertain mean and covariance of the asset return with uncertainty sets $\mathcal{U}_{\mu t}$ and $\mathcal{U}_{Q t}$. Problem \eqref{eq:mp_port} maximizes the total profit from total risk-adjusted expected return $s_t\in\mathbb{R}$ excluding the transaction cost $c_t\in\mathbb{R}$ under the uncertainty of mean and covariance. The left-hand side of \eqref{eq:port_muQ_t} gives the risk-adjusted expected return $s_t$ in period $t$. The left-hand side of \eqref{eq:port_transac} gives the transaction cost $c_t$ associated with the changes on allocation decision $x_t$. The problem contains $T$ robust constraints with $n(n+1)$ uncertainty variables per robust constraint. 
           \begin{subequations}\label{eq:mp_port}
           \begin{align}
          \max_{x_t\geq 0,\ s_t,\ c_t\geq 0}\quad & \sum_{t=1}^T (s_t - c_t) \label{eq:port_obj_t}\\
          \text{s.t}\quad& \mu_t^\top x_t - \lambda x_t^\top Q_t x_t\geq s_t,\ \forall\ \mu_t\in\mathcal{U}_{\mu_t},\ \forall Q_t\in\mathcal{U}_{Q_t},\  \forall t\in\{1,...,T\},\label{eq:port_muQ_t}\\
          &\sum_{i=1}^n x_{it}=1,\ \forall t\in\{1,...,T\},\\
          & U\|x_{t+1}-x_{t}\|_1\leq c_t,\ \forall t\in\{1,...,T\},\label{eq:port_transac}.
          \end{align}
          \end{subequations}
Further, we define uncertainty sets $\mathcal{U}_{\mu t}:=\{\underline{\mu}_{it}\leq\mu_{it}\leq\overline{\mu}_{it},\ i=1,...,n\}$ and $\mathcal{U}_{Qt}:=\{{\color{black}Q\succeq 0},\ \sum_{i,j}\frac{(Q_{ijt}-C_{ijt})^2}{r_{ijt}^2}\leq 1,\ i=1,...,n,\ j=1,...,n,\}$. $C_{ijt}$ and $r_{ijt}$ are the corresponding center and radius for entry $Q_{ijt}$. 
The uncertainty set $\mathcal{U}_{\mu t}$ and $\mathcal{U}_{Qt}$ are constructed by using the Australian stock price dataset \cite{port_data}. In our numerical study we choose three cases among $T\in\{7, 14, 21\}$ and $n = 2$ (i.e., a total uncertainty dimension of $6T$) with transaction cost selection $U\in\{0.05,0.15,0.25,0.35\}$. Because $\lambda_t$ does not interfere with the problem dimension, we use $\lambda_t = 1$ for simplicity in these computational tests. For each given horizon and asset setting, we randomly select $3$ groups of assets and demonstrate the average computational performance across the groups. 

% construction and result collection follows the same strategy as single-period asset allocation.     

\begin{table}[ht]
\setlength{\tabcolsep}{.16667em}
     \caption{Average iteration and runtime of Polak's and the superset algorithms for $T=\{7, 14, 21\}$.}
\centering
    \begin{subtable}[h]{1\textwidth}
       \caption{Runtime (Sec).}
    \centering
 \begin{tabular}{c|cccccc}
\hline
 & \multicolumn{6}{c}{Horizon}                                                                                                                                                                                                                                          \\ 
                                                                              & \multicolumn{2}{c}{$T = 7$}                                                                                                                         & \multicolumn{2}{c}{$T = 14$} &  \multicolumn{2}{c}{$T = 21$}                                                                                                \\ \hline \hline
                                                                            \begin{tabular}[c]{@{}c@{}}Transaction\\ Cost $U$\end{tabular}  & \begin{tabular}[c]{@{}c@{}}Polak's \\ Algorithm\end{tabular} & \multicolumn{1}{c|}{\begin{tabular}[c]{@{}c@{}}Superset \\ Algorithm\end{tabular}} & \begin{tabular}[c]{@{}c@{}}Polak's \\ Algorithm\end{tabular} & \multicolumn{1}{c|}{\begin{tabular}[c]{@{}c@{}}Superset \\ Algorithm\end{tabular}}
                                                                            & \begin{tabular}[c]{@{}c@{}}Polak's \\ Algorithm\end{tabular} & \begin{tabular}[c]{@{}c@{}}Superset \\ Algorithm\end{tabular}
                                                                            \\ \hline
0.05                                                                             & 1.23                                                          & \multicolumn{1}{c|}{1.55}                                                              & 2.22                                                          & \multicolumn{1}{c|}{1.82} & 4.48 & 1.83                                                             \\ \hline
0.15                                                                            & 0.20                                                          & \multicolumn{1}{c|}{0.13}                                                              & 1.08                                                          & \multicolumn{1}{c|}{0.31} & 2.63 & 0.35                                                             \\ \hline
0.25                                                                            & 0.20                                                          & \multicolumn{1}{c|}{0.11}                                                              & 1.08                                                          & \multicolumn{1}{c|}{0.22} & 2.54 & 0.58                                                            \\ \hline
0.35                                                                            & 0.08                                                          & \multicolumn{1}{c|}{0.06}                                                              & 0.27                                                          & \multicolumn{1}{c|}{0.13} & 0.68 & 0.28                                                            \\ \hline
\end{tabular}
       \label{tab:lowh_lows2}
    \end{subtable}
    \newline
    \vspace{0.3cm}
    \newline
    \begin{subtable}[h]{1\textwidth}
    \caption{Iteration Count.}
        \centering
\begin{tabular}{c|cccccc}
\hline
 & \multicolumn{6}{c}{Horizon}                                                                                                                                                                                                                                          \\ 
                                                                              & \multicolumn{2}{c}{$T = 7$}                                                                                                                         & \multicolumn{2}{c}{$T = 14$} &  \multicolumn{2}{c}{$T = 21$}                                                                                                \\ \hline \hline
                                                                            \begin{tabular}[c]{@{}c@{}}Transaction\\ Cost $U$\end{tabular}  & \begin{tabular}[c]{@{}c@{}}Polak's \\ Algorithm\end{tabular} & \multicolumn{1}{c|}{\begin{tabular}[c]{@{}c@{}}Superset \\ Algorithm\end{tabular}} & \begin{tabular}[c]{@{}c@{}}Polak's \\ Algorithm\end{tabular} & \multicolumn{1}{c|}{\begin{tabular}[c]{@{}c@{}}Superset \\ Algorithm\end{tabular}}
                                                                            & \begin{tabular}[c]{@{}c@{}}Polak's \\ Algorithm\end{tabular} & \begin{tabular}[c]{@{}c@{}}Superset \\ Algorithm\end{tabular}
                                                                            \\ \hline
0.05                                                                             & 3.0                                                          & \multicolumn{1}{c|}{7.0}                                                              & 4.3                                                          & \multicolumn{1}{c|}{5.7} & 7.0 & 4.0                                                             \\ \hline
0.15                                                                            & 3.0                                                          & \multicolumn{1}{c|}{3.3}                                                              & 3.0                                                          & \multicolumn{1}{c|}{3.7} & 3.7 & 2.3                                                             \\ \hline
0.25                                                                            & 3.0                                                          & \multicolumn{1}{c|}{3.3}                                                              & 4.7                                                          & \multicolumn{1}{c|}{2.3} & 4.7 & 4.0                                                            \\ \hline
0.35                                                                            & 3.0                                                          & \multicolumn{1}{c|}{1.3}                                                              & 3.0                                                          & \multicolumn{1}{c|}{1.3} & 3.0 & 2.7                                                            \\ \hline
\end{tabular}
        \label{tab:lowh_highs2}
     \end{subtable}
     \label{tab:port_low2}
\end{table}

The results are summarized in Table~\ref{tab:port_low2}. We observe that Polak's algorithm requires more iterations and runtime compared to the superset algorithm when the horizon increases (i.e., more robust constraints in \eqref{eq:mp_port}). As the number of robust constraints increases, more nonlinear constraints are added in Polak's subproblem \eqref{eq:Polak_fin} as iteration goes, while the constraint dimension of the superset algorithm remains the same. As a result, when the horizon increases, the superset algorithm demonstrates a steady computational performance and an increasing advantage over the Polak's algorithm. In terms of the transaction cost change, when $U$ increases, we observe reductions on the iteration count and runtime for both algorithms.

To summarize the computational difference between Polak's algorithm and the superset algorithm, we provide the performance profile \cite{performprofile} of the CPU runtime in Fig.~\ref{fig:performance_profile_all}). We conclude that the superset algorithm outperforms Polak's algorithm in $92\%$ of the instances. Specifically, for $52\%$ of the cases, the superset algorithm is more than 2.0 times faster than the Polak's algorithm; for $27\%$ of the cases, the superset algorithm is more than 3 times faster; for $17\%$ of the cases, the superset algorithm is more than 4 times faster.

\begin{figure}[H]
\centering
\includegraphics[width=3.5in]{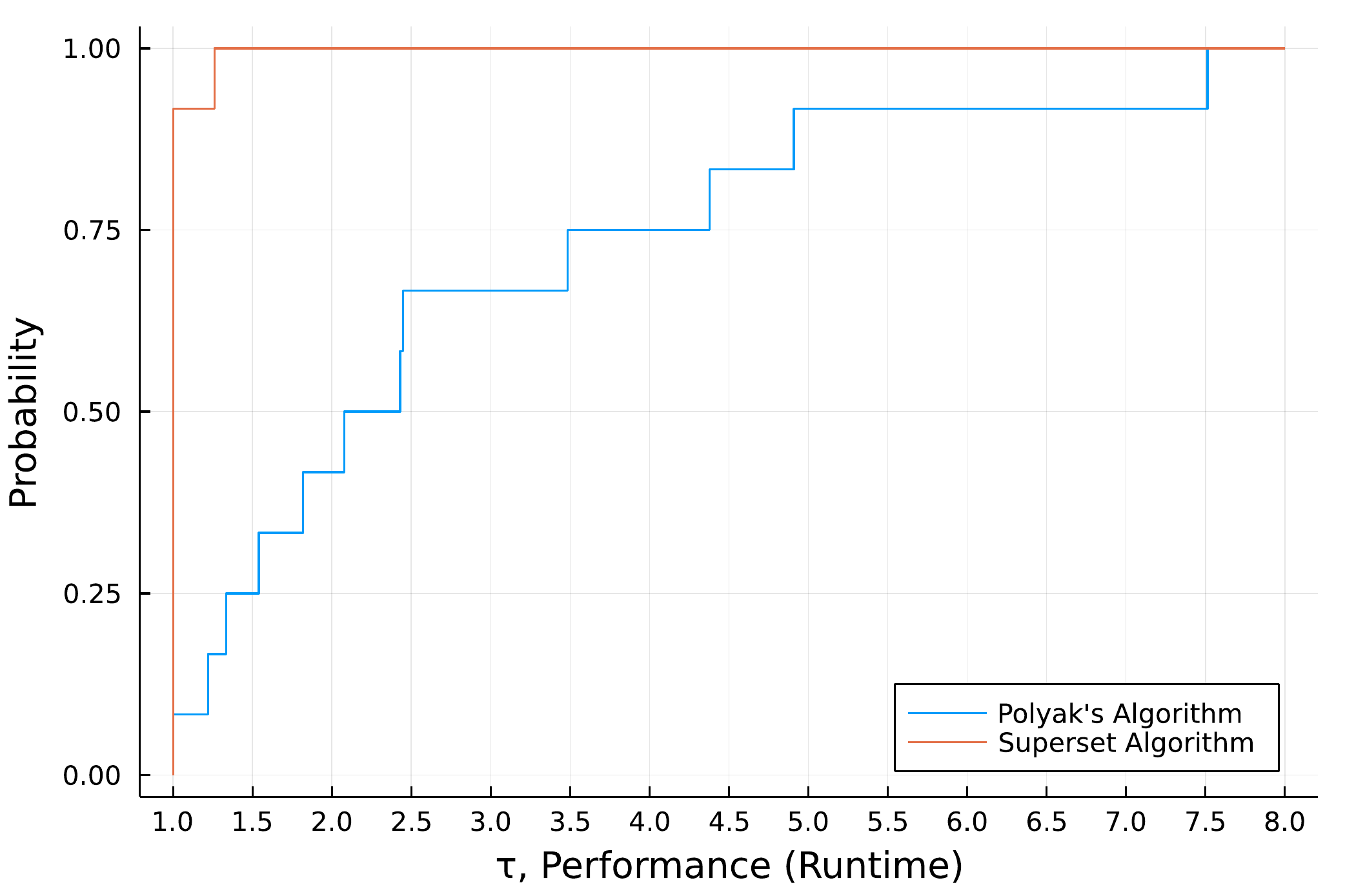}
\caption{Performance profile of CPU runtime comparison: the superset algorithm wins $92\%$ of the multiperiod portfolio optimization instances.}
\label{fig:performance_profile_all}
\end{figure}

\subsection{Power System Application: Production Cost }\label{sec:prod_cost}
We consider a production cost minimization problem  for a given horizon $t\in\{1,...,T\}$ between two generators with well-forecast load $d_t$ and uncertain cost \cite{prod_cost}. We assume the first generator is expensive with ramping costs and limits, whereas the second generator is cheap and has no ramping limitation. For each period $t$, we denote the productions of the first generator as $x_t$ with cost $c_{1,t}$ and the ramp rate $u_t$ with control cost $c_{3,t}$. Because of power balance at each period $t$, we denote the production of the second generator as $d_t-x_t$ with cost $c_{2,t}$. For each $t$, we further define the uncertainty set as $\mathcal{C}_t$. The full problem is formulated as follows: 
    \begin{subequations}
     \begin{align}
        \min_{x_{t},\ u_{t},\ s_{t}}\quad &\sum_{t=1}^T s_t\\
        \text{s.t}\quad &c_{1,t}(x_t-d_t)^2 + c_{2,t}x_t^2 \notag\\
        &\quad\quad +c_{3,t}u_t^2\leq s_t, \ \forall(c_{1,t},c_{2,t},c_{3,t})\in\mathcal{C}_t,\forall t\in\{1,...,T\}\label{eq:prod_obj}\\
        &x_{t+1}=x_t+u_t,\ \forall t\in\{1,...,T\}, \label{eq:ramp_1}\\
        & -U\leq u_t\leq U,\ \forall t\in\{1,...,T\}, \label{eq:ramp_2}
    \end{align}
    \end{subequations}
where the left-hand side of \eqref{eq:prod_obj} represents the cost at time $t$ as $s_t$ and \eqref{eq:ramp_1} and \eqref{eq:ramp_2} represent the ramp dynamic and limit $U$, respectively. In the simulation we fix $T=24$, and the problem contains $73$ variables with $24$ robust constraints and $48$ deterministic constraints. 
    
For $d_t$, we use the seven-day load forecast dataset from PJM \cite{pjm_data} and select three daily load-forecast patterns from Dec. 2021 (see Fig.~\ref{fig:load_pattern}). For uncertainty sets $\mathcal{C}_t$, we first define the nominal costs $o_{1}=10$, $o_{2}=5$, $o_{3}=1$ and then construct ellipsoidal uncertainty sets as $\frac{(c_{1,t}-o_1)^2}{r_1^2}+\frac{(c_{2,t}-o_2)^2}{r_2^2}+\frac{(c_{3,t}-o_3)^2}{r_3^2}\leq 1$ with $r_{i} = po_{i},\ i=1,2,3$, where $p\in(0,1]$ is a percentage parameter controlling the size of the ellipsoid. As the comparison setup, we use 3 load patterns, 3 different values of $U\in\{80,120,160\}$, and 3 different values of $p\in\{0.7,0.8,0.9\}$ and in total run 27 instances with every combination. 

% The load patterns are shown in Fig.~\ref{fig:load_pattern}.

\begin{figure}[htb]
\centering
\includegraphics[width=3.5in]{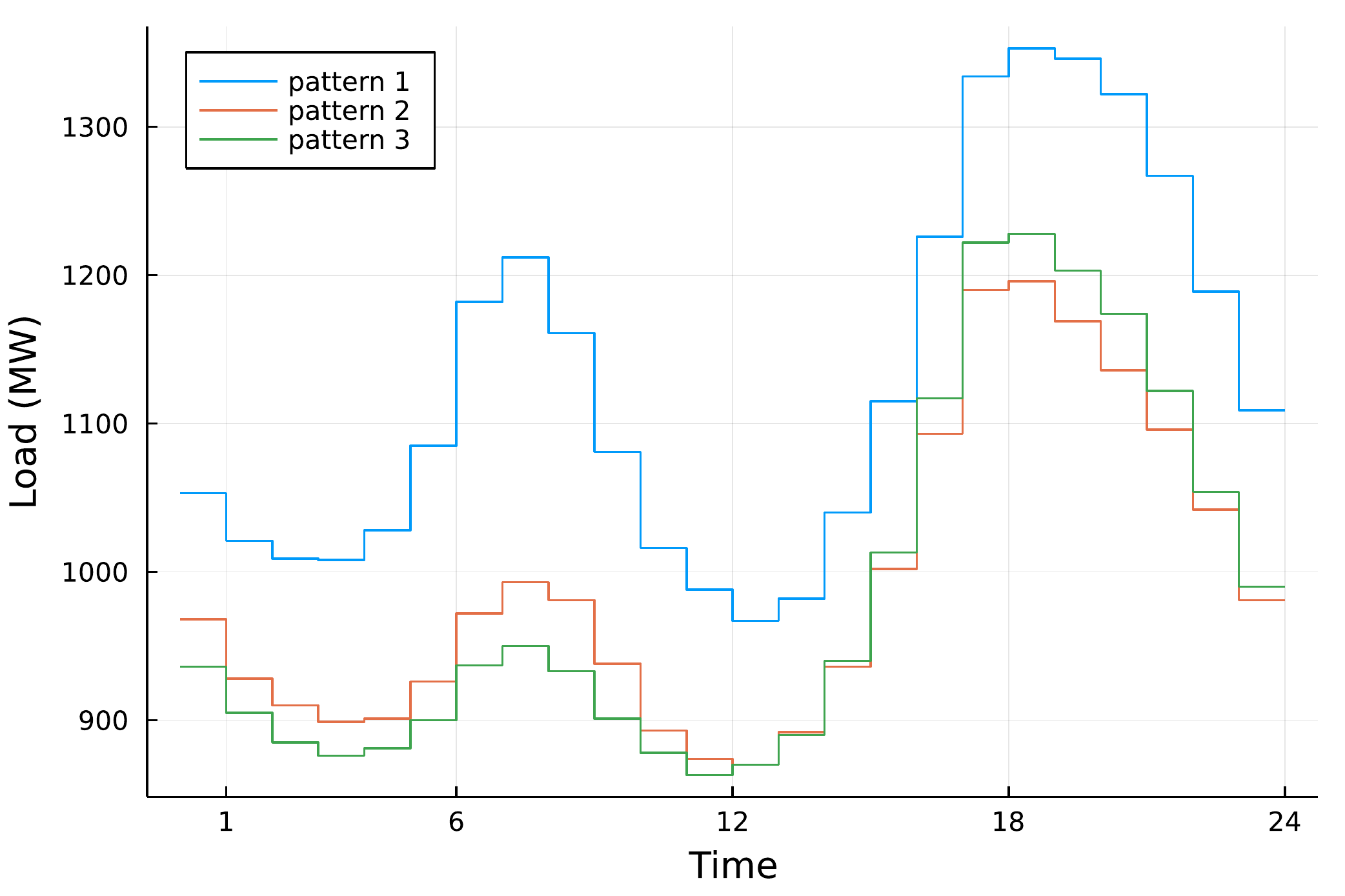}
\caption{Three daily load patterns used in the testing.}
\label{fig:load_pattern}
\end{figure}

First, we show the iteration count comparison. The iteration count of all instances is summarized in Table~\ref{tab:iter}. In all cases, the superset algorithm takes fewer iterations than Polak's algorithm does. Both methods have stable iteration count regardless of the  changes on load pattern, ramp limits, and the size of the uncertainty set.

\begin{table}[htb]
\caption{Iteration count comparison for all testing instances.}
\centering
\resizebox{\columnwidth}{!}{
\begin{tabular}{c|c|cccccc}
\hline
 &  & \multicolumn{6}{c}{Ramp Rate $U$}                                                          \\  
                              &                                          & \multicolumn{2}{c}{$80$}    & \multicolumn{2}{c}{$120$}    & \multicolumn{2}{c}{$160$} \\ \hline\hline 
                 \begin{tabular}{c}Load \\Pattern\end{tabular}             &    \begin{tabular}{c}Uncertainty\\ Set Size $p$\end{tabular}                                      & \begin{tabular}[c]{@{}c@{}}Polak's \\ Algorithm\end{tabular} & \multicolumn{1}{c|}{\begin{tabular}[c]{@{}c@{}}Superset \\ Algorithm\end{tabular}} & \begin{tabular}[c]{@{}c@{}}Polak's \\ Algorithm\end{tabular} & \multicolumn{1}{c|}{\begin{tabular}[c]{@{}c@{}}Superset \\ Algorithm\end{tabular}} & \begin{tabular}[c]{@{}c@{}}Polak's \\ Algorithm\end{tabular}          & \begin{tabular}[c]{@{}c@{}}Superset \\ Algorithm\end{tabular}          \\ \hline
1                            & $0.7$                                    & 10  & \multicolumn{1}{c|}{6}  & 10  & \multicolumn{1}{c|}{6}  & 10           & 6           \\
1                              & $0.8$                                    & 10  & \multicolumn{1}{c|}{6}  & 10  & \multicolumn{1}{c|}{6}  & 10           & 6           \\
1                             & $0.9$                                    & 10 & \multicolumn{1}{c|}{6} & 10 & \multicolumn{1}{c|}{6}  & 10          & 6           \\ \hline
2                             & $0.7$                                    & 10  & \multicolumn{1}{c|}{6}  & 10 & \multicolumn{1}{c|}{6}  & 10           & 6           \\
 2                            & $0.8$                                    & 10  & \multicolumn{1}{c|}{6}  & 10  & \multicolumn{1}{c|}{6}  & 10           & 6           \\
 2                             & $0.9$                                    & 10 & \multicolumn{1}{c|}{6} & 10 & \multicolumn{1}{c|}{6}  & 10          & 6           \\ \hline
 3                            & $0.7$                                    & 10  & \multicolumn{1}{c|}{6}  & 10  & \multicolumn{1}{c|}{6}  & 10           & 6           \\
 3                            & $0.8$                                    & 10  & \multicolumn{1}{c|}{6} & 10  & \multicolumn{1}{c|}{6}  & 10           & 6           \\
 3                            & $0.9$                                    & 10 & \multicolumn{1}{c|}{6} & 10 & \multicolumn{1}{c|}{6}  & 10          & 6           \\\hline
\end{tabular}
}
\label{tab:iter} 
\end{table}

% due the different subproblem dimension and iteration progress between the algorithms

Next, we show the runtime comparison. The runtimes for all instances are summarized in Table~\ref{tab:time}.  In all cases, superset algorithm takes less CPU time than Polak's algorithm does. Given the stable iteration count in Table~\ref{tab:iter}, superset algorithm also has stable CPU time regardless of the changes on test settings. On the other hand, the CPU time of Polak's algorithm is greatly dependent on the load patterns, ramp limits and the size of uncertainty set but no general monotonic relation can be observed.   

\begin{table}[htb]
\caption{Runtime (sec) comparison for all testing instances.}
\centering
\resizebox{\columnwidth}{!}{
\begin{tabular}{c|c|cccccc}
\hline
 &  & \multicolumn{6}{c}{Ramp Rate $U$}                                                          \\ 
                              &                                          & \multicolumn{2}{c}{$80$}    & \multicolumn{2}{c}{$120$}    & \multicolumn{2}{c}{$160$} \\ \hline\hline  \begin{tabular}{c}Load \\Pattern\end{tabular}
                              &      \begin{tabular}{c}Uncertainty\\ Set Size $p$\end{tabular}                                    & \begin{tabular}[c]{@{}c@{}}Polak's \\ Algorithm\end{tabular} & \multicolumn{1}{c|}{\begin{tabular}[c]{@{}c@{}}Superset \\ Algorithm\end{tabular}} & \begin{tabular}[c]{@{}c@{}}Polak's \\ Algorithm\end{tabular} & \multicolumn{1}{c|}{\begin{tabular}[c]{@{}c@{}}Superset \\ Algorithm\end{tabular}} & \begin{tabular}[c]{@{}c@{}}Polak's \\ Algorithm\end{tabular}          & \begin{tabular}[c]{@{}c@{}}Superset \\ Algorithm\end{tabular}          \\ \hline
1                            & $0.7$                                    & 0.767  & \multicolumn{1}{c|}{0.748}  & 1.653  & \multicolumn{1}{c|}{0.744}  & 1.430           & 0.742           \\
1                              & $0.8$                                    & 0.873  & \multicolumn{1}{c|}{0.715}  & 1.913  & \multicolumn{1}{c|}{0.725}  & 0.972           & 0.716           \\
1                             & $0.9$                                    & 0.806 & \multicolumn{1}{c|}{0.731} & 1.271 & \multicolumn{1}{c|}{0.760}  & 1.000          & 0.750           \\ \hline
2                             & $0.7$                                    & 1.230  & \multicolumn{1}{c|}{0.715}  & 1.095 & \multicolumn{1}{c|}{0.743}  & 1.704           & 0.737           \\
 2                            & $0.8$                                    & 1.192  & \multicolumn{1}{c|}{0.703}  & 1.244  & \multicolumn{1}{c|}{0.711}  & 1.480           & 0.706           \\
 2                             & $0.9$                                    & 0.766 & \multicolumn{1}{c|}{0.755} & 1.313 & \multicolumn{1}{c|}{0.749}  & 1.323          & 0.750           \\ \hline
 3                            & $0.7$                                    & 1.703  & \multicolumn{1}{c|}{0.721}  & 1.139  & \multicolumn{1}{c|}{0.721}  & 1.126           & 0.727           \\
 3                            & $0.8$                                    & 0.908  & \multicolumn{1}{c|}{0.708} & 1.617  & \multicolumn{1}{c|}{0.696}  & 1.301           & 0.684           \\
 3                            & $0.9$                                    & 1.112 & \multicolumn{1}{c|}{0.748} & 1.019 & \multicolumn{1}{c|}{0.745}  & 1.729          & 0.740           \\\hline
\end{tabular}
}
\label{tab:time}
\end{table}

To summarize the computational difference between Polak's algorithm and the superset algorithm, we provide the performance profile \cite{performprofile} of the CPU runtime in Fig.~\ref{fig:performance_profile_time}. We observe that, in all cases, the superset algorithm outperforms Polak's algorithm. Specifically, for $59\%$ of the cases, superset algorithm is more than 1.5 times faster than the Polak's algorithm; for $22\%$ of the cases, superset algorithm is more than 2 times faster than the Polak's algorithm.  

\begin{figure}[H]
\centering
\includegraphics[width=3.5in]{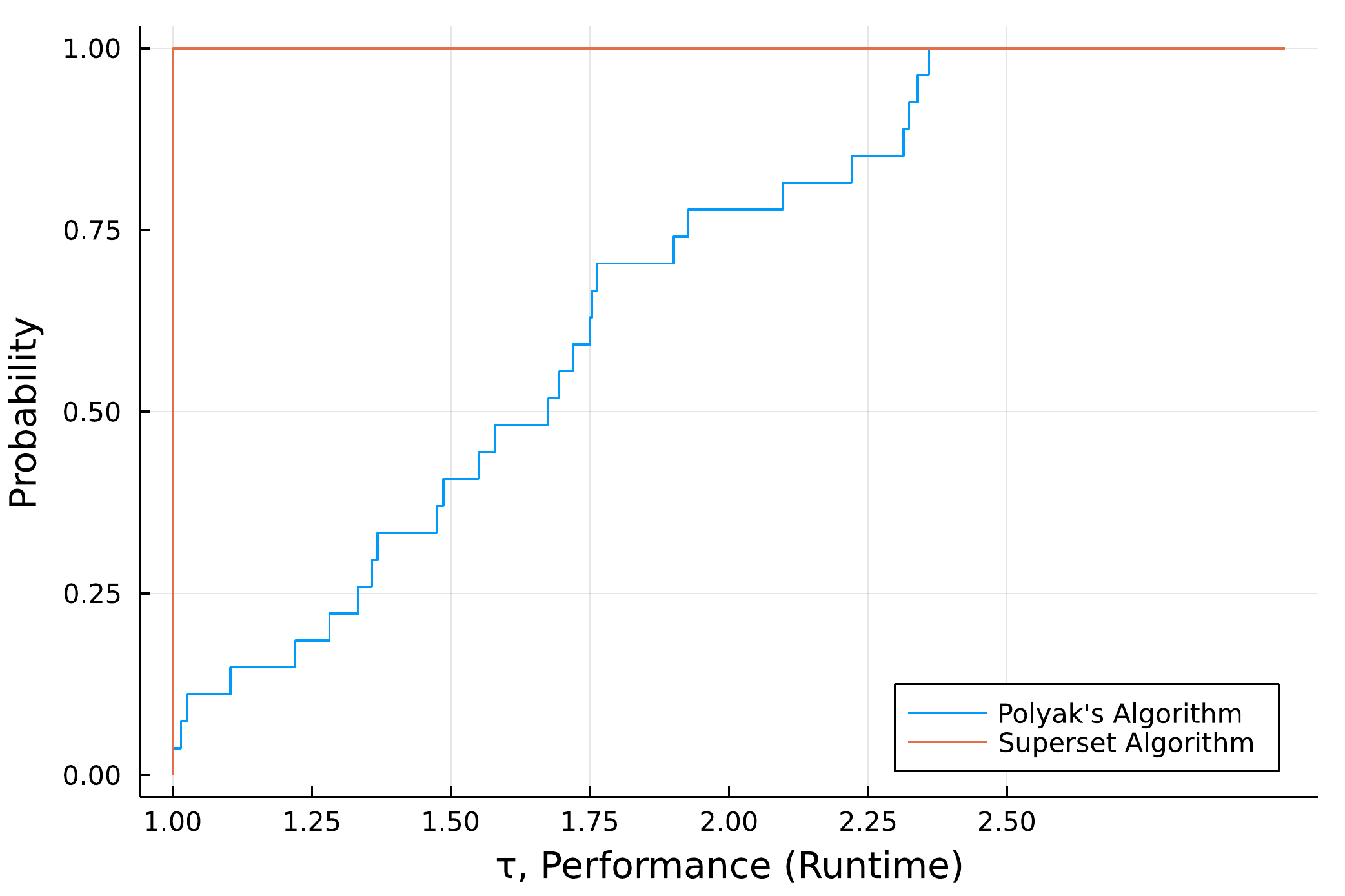}
\caption{Performance profile of CPU runtime comparison: the superset algorithm is superior in all the production cost minimization instances.}
\label{fig:performance_profile_time}
\end{figure}

\section{Conclusion}\label{sec:conclusion}

We have developed a superset algorithm for a class of structured NRO problems. The algorithm iteratively solves the reformulation of an NRO subproblem with the polytopic supersets of the uncertainty set. Different cutting plane methods are proposed to improve the supersets over iteration. We showed that the solution iterates from the superset algorithm are feasible in the original NRO problem and provide both lower and upper bounds to the optimal objective value. We proved the convergence of the superset algorithm under the assumption that uncertainty sets are convex. We also provided a feasibility restoration algorithm to detect whether the NRO is infeasible or restore the feasibility of the NRO subproblem of the superset algorithm by constructing a new superset.  

To evaluate the computational performance, we compared the superset algorithm with Polak's algorithm in applications including portfolio optimization and production cost minimization. We demonstrated that the superset algorithm is more advantageous than Polak's algorithm when the number of robust constraints is large.

For future work, we plan to extend the superset algorithm to more general NRO formulations by relaxing the current structural assumptions of the affine relationship between the constraint and uncertainty parameter. 

\backmatter

% \bmhead{Supplementary information}

% If your article has accompanying supplementary file/s please state so here. 

% Authors reporting data from electrophoretic gels and blots should supply the full unprocessed scans for key as part of their Supplementary information. This may be requested by the editorial team/s if it is missing.

% Please refer to Journal-level guidance for any specific requirements.

\bmhead{Acknowledgments}

This work was supported by the Applied Mathematics activity within the U.S. Department of Energy, Office of Science, Advanced Scientific Computing Research, under Contract DE AC02-06CH11357.

\bibliography{sn-bibliography}% common bib file
%% if required, the content of .bbl file can be included here once bbl is generated
%%\input sn-article.bbl

%% Default %%
%%\input sn-sample-bib.tex%

\end{document}